\documentclass[a4paper,12pt]{article}
\usepackage{makeidx}
\usepackage[english]{babel}
\usepackage[latin1]{inputenc}
\usepackage{amsfonts}
\usepackage{bm}
\usepackage{amssymb}
\usepackage{latexsym}
\usepackage{amsmath}
\usepackage{amscd}
\usepackage{amsthm}
\usepackage{rotating}
\usepackage{graphicx}
\usepackage{pifont}
\usepackage{curves}
\usepackage{fancyhdr}
\usepackage{epsfig}
\usepackage{pstricks}
\usepackage{pst-tree}
\usepackage{subfigure}
\usepackage{epic}
\usepackage{alltt}
\usepackage[all]{xy}

\newcommand{\Co}{\mathcal{C}}

\newcommand{\N}{\mathbb{N}}

\newcommand{\R}{\mathbb{R}}

\newcommand{\tlangle}{< \! \! \! \mid}
\newcommand{\trangle}{\mid \! \! \! >}

\newtheorem{theorem}{\sc Theorem}[section]
\newtheorem{lemma}{\sc Lemma}[section]
\newtheorem{proposition}{\sc Proposition}[section]
\newtheorem{remark}{\sc Remark}
\newtheorem{definition}{\sc Definition}[section]

\def\R{{\mathbb R}}
\def\N{{\mathbb N}}

\begin{document}

\title{A stochastic SIR model with contact-tracing:\\ large population limits and statistical inference}

\author{St\'ephan Cl\'{e}men\c{c}on, Viet Chi Tran, and Hector de Arazoza}

\maketitle

\begin{abstract}
This paper is devoted to present and study a specific stochastic epidemic model accounting for the effect of \textit{contact-tracing} on the spread of an infectious disease. Precisely, one considers here the situation in which individuals identified as infected by the public health detection system may contribute to detecting other infectious individuals by  providing information related to persons with whom they have had possibly infectious contacts. The control strategy, that consists in examining each individual  one has been able to identify on the basis of the information collected within a certain time period, is expected to reinforce efficiently the standard random-screening based detection and slack considerably the epidemic. In the novel modelling of the spread of a communicable infectious disease considered here, the population of interest evolves through demographic, infection and detection processes, in a way that its temporal evolution is described by a stochastic Markov process, of which the component accounting for the contact-tracing feature is assumed to be valued in a space of \textit{point measures}. For adequate scalings of the demographic, infection and detection rates, it is shown to converge to the weak deterministic solution of a PDE system, as a parameter $n$, interpreted as the population size roughly speaking, becomes large. From the perspective of the analysis of infectious disease data, this approximation result may serve as a key tool for exploring the asymptotic properties of standard inference methods such as maximum likelihood estimation. We state preliminary statistical results in this context. Eventually, relation of the model to the available data of the HIV epidemic in Cuba, in which country a contact-tracing detection system has been set up since 1986, is investigated and numerical applications are carried out.
\end{abstract}



\section{Introduction}\label{intro}

Since the seminal contribution of \cite{kermackmckendrick}, the mathematical modelling of epidemiological phenomena has received increasing attention in the applied mathematics literature. References devoted to epidemic modelling or statistical analysis of infectious disease data are much too numerous for being listed in this paper (refer to \cite{andersonbritton, mode} for recent accounts of stochastic epidemic modelling, while deterministic models for the spread of infectious diseases are comprehensively presented and discussed in \cite{castillochavez2002, AndersonMay}). Here, an attempt is made to extend the 'general stochastic epidemic model', usually referred to as the standard SIR model, in order to take appropriate account of the effect of a \textit{contact-tracing} control measure on the spread of the epidemic at the population level and to acquire a better understanding of the efficiency of this intervention strategy.
\par
In the area of public health practice, by \textit{contact-tracing} one means the active detection mechanism that consists in asking individuals identified as infected to name persons  with whom they have had possibly infectious contacts and then, on the basis of the information provided, in striving to find those persons in order to propose them a medical examination and a cure in the event of infection. Though expensive and controversial, 'contact-tracing' programs are now receiving much attention both in the scientific literature (see \cite{Hyman, Hector, Rapatski} or \cite{Muller} for instance) and in the public health community, within which they are generally considered as efficient guidance methods for bringing the spread of sexually transmissible diseases (STD's) under control. For instance, a contact-tracing detection system has been set up since 1986 for controlling the HIV epidemic in Cuba (refer to \cite{Auvert} for an overview of the evolution of HIV/AIDS in Cuba), which shall serve as a running illustration for the concepts and methods studied in the present paper. From the perspective of public health guidance practice, mathematical modelling of epidemics in presence of a contact-tracing strategy reinforcing a screening-based detection system is a crucial stake, insofar as it may help evaluating the impact of this costly control measure. In this framework, epidemic models must naturally account for the fact that, once detected, an infected person keeps on playing a role in the evolution of the epidemic for a certain time by helping towards identification of infectious individuals.
\par
The primary goal of this paper is to generalize the standard SIR model by incorporating a \textit{structure by age} in the subpopulation of detected individuals, age being here the time since which a person has been identified as infected. At any time, the 'R' class is described by a \textit{point measure}, on which the contact-tracing detection rate is supposed to depend. In this manner, the way an 'R' individual contributes to contact-tracing detection may be made strongly dependent on the time since her/his detection through a given \textit{weight function} $\psi$, allowing for great flexibility in the modelling. Assuming in particular a large population in which the infectious disease is spread, properties of the mathematical model are thoroughly investigated and preliminary statistical questions are tackled.  Beyond stochastic modelling of the contact-tracing feature, the present work establishes large population limit results (law of large numbers and central limit theorem) for the measure-valued Markov process describing the epidemic (we follow in this respect the approach developed in \cite{fourniermeleard, meleardfluctuation, chithese, trangdesdev}), as well as in its application to statistical analysis of the epidemic.
\par The paper is organized as follows. In Section \ref{sectionmodel}, a Markov process with an age-structured component is introduced for modelling the temporal evolution of an epidemic in presence of contact-tracing. A short qualitative description is provided, aiming at giving an insight into how the dynamic is driven by a few key components. The process of interest is the solution of a stochastic differential equation (SDE) for which existence and uniqueness results are stated, together with a short probabilistic study. The main results of the paper are displayed in Section 3. Considering a sequence of epidemic models with contact-tracing indexed by a parameter $n\in \mathbb{N}^*$ representing the population size roughly speaking, limit results are established when $n\rightarrow \infty$. Applications of the latter to the study of maximum likelihood estimators (based on complete data) in the context of statistical parametric estimation of the epidemic model with contact-tracing are then considered in Section 4. Eventually, in Section 5, these inference techniques are applied for analyzing data related to the HIV epidemic in Cuba and drawing preliminary conclusions about the effectiveness of contact-tracing in this particular case: it can be seen that the chosen model especially reflects the growing efficiency of the contact-tracing detection method, the latter becoming almost as competitive as the random screening based method ten years after the beginning of the epidemic. Technical proofs are postponed to the Appendix.

\section{The stochastic SIR model with contact-tracing}\label{sectionmodel}

Epidemic problems really present a great challenge to probabilists and statisticians. Models for the spread of infections are based on hypotheses about such
mechanisms as infection and detection. The huge diversity of possible hypotheses could give rise to an enormous variety of probabilistic models with
their specific features. Although comprehensive mathematical models should incorporate numerous features to account for real-life situations (such as population stratified according to socio-demographic characteristics, time-varying infectivity,
effects of latent period, change in behavior, \textit{etc.}), we shall deal with
a stochastic epidemic model with a reasonably simple structure (a modified version of
the standard 'Markovian SIR model with demography', actually), while
covering some important aspects and keeping thus its pertinence from the perspective of practical applications. Indeed, incorporating too many features would naturally make the model too
difficult to study analytically. As previously mentioned, we are mainly concerned here with a probabilistic modelling of the spread of an infectious disease in presence of a contact-tracing control strategy in the long-range (\textit{i.e.} when one cannot assume that the epidemic ceases before some demographic changes occur,
leading up to take into consideration immigration/birth and emigration/death processes). A parcimonious Markovian structure for describing these features is stipulated, the main novelty arising from the measure-valued component incorporated into the model in order to account for the effects of contact-tracing. Beyond its simplicity, our modelling hopefully suffices to shed some light on the problem of investigating the efficacy of such a control measure. In the case of the HIV epidemic in Cuba for instance (see Section 5), the model obtained accounts for the fact that contact-tracing has become as efficient as random screening after 6 years. To our knowledge, earlier works have not allowed to construct a model reflecting this phenomenon (see \cite{Hector} and the references therein).

\subsection{The population dynamics} \label{dynamics}

We start with a qualitative description of the
population dynamics and a list of all possible events through which the
population of interest may evolve (see Fig. 1). The population is structured into three classes corresponding to the different possible states with respect to the infectious disease. We adopt the standard \textit{SIR} terminology for denoting the current status of an individual with the only differences that '\textit{R}' stands here for the population of 'removed individuals willing to take part in the contact-tracing program' and that it is structured according to the \textit{age of detection}, namely the time since a detected individual has been identified by the public health detection system as infected. Such a distinction allows for considering heterogeneity in the way each 'R' individual contributes to the contact-tracing control. Hence, at any time $t\geq 0$ the class of removed individuals is described by $R_t(da)$ in $\mathcal{M}_P(\R_+)$, the set of point measures on $\mathbb{R}_+$: for all $0<a_1<a_2<\infty$, the quantity $R_t([a_1,a_2])$ represents the number of removed
  individuals who have been detected between times $t-a_2$ and $t-a_1$. Here and throughout, we use the notation $\langle R, \psi \rangle=\int \psi(a)R(da)$, $R$ being any positive measure on $\mathbb{R}_+$ and $\psi$ any $R$-integrable function. In a more standard fashion, we shall denote by $S_t$ and $I_t$ the sizes of the classes of susceptible and infectious individuals.

\unitlength=1cm
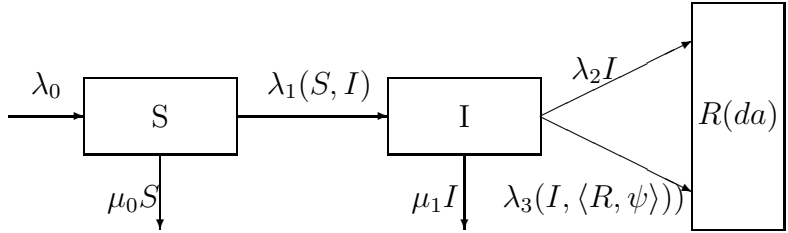
\begin{figure}[ht]
  \begin{center}
    \begin{picture}(13,3)
    \put(1,1.5){\vector(1,0){1}}
    \put(1.3,1.8){$\lambda_0$}
    \put(2,1){\framebox(2,1){S}}
    \put(3,1){\vector(0,-1){1}}
    \put(2.3,0.3){$\mu_0 S$}
    \put(4,1.5){\vector(1,0){2}}
    \put(4.4,1.8){$\lambda_1(S,I)$}
    \put(6,1){\framebox(2,1){I}}
    \put(7,1){\vector(0,-1){1}}
    \put(6.3,0.3){$\mu_1 I$}
    \put(8,1.5){\vector(2,1){2}}
    \put(8,1.5){\vector(2,-1){2}}
    \put(8.4,2){$\lambda_2 I$}
    \put(7.5,0.3){$\lambda_3(I,\langle R,\psi\rangle))$}
    \put(10,0){\framebox(1.2,3){$R(da)$}}
    \end{picture}
  \caption{\textit{Flow-diagram of the SIR model with contact-tracing detection.}}
  \end{center}
\end{figure}

\noindent Individuals
immigrate one at a time according to a Poisson process
of intensity $\lambda_{0}$. Once in the population, an individual becomes 'susceptible' and may either
leave the population without being contaminated (emigration or death) or independently be infected. Emigrations occur in the population at time $t\geq 0$ with the hazard rate $\mu_0 S_t$ and infections with the rate $\lambda_{1}(S_t,I_t)$. Once
infected, an individual can be discovered by the detection system either by random screening ('spontaneous detection') or by contact-tracing, or else emigrates/dies. The hazard rates associated with these events are respectively $\lambda_{2} I_t$, $\lambda_{3}(I_t,\langle R_t, \psi\rangle)$, where $\psi:\mathbb{R}_+\rightarrow \mathbb{R}_+$ is a bounded and measurable \textit{weight function} that determines the contribution of a removed individual to the contact-tracing control according to the time $a$ she/he has been detected (see the examples discussed below) and $\mu_1 I_t$.
If detected, an individual takes part in the contact-tracing system by providing useful information related to her/his (possibly) infectious contacts. We do not consider the emigration/death of detected individuals since it is the availability of the information that they have given rather than their presence in the system that plays a role in the contact-tracing process. In order to avoid possible misunderstanding due to the notation, we underline that $\lambda_1(.,.)$ and $\lambda_3(.,.)$ here denote jump rate functions related to the SIR process and not the individual rates (for instance, given the class sizes $S$ and $I$, the infection rate of a given susceptible is $\lambda_1(S,I)/S$).
\par The events through which the sizes $S_t$, $I_t$ and the point measure $R_t$ evolve are numbered as follows:
\begin{itemize}
\item Event $E=0$: recruitment of a susceptible,
\item Event $E=1$: death/emigration of a susceptible,
\item Event $E=2$: infection,
\item Event $E=3$: 'spontaneous' detection of an infective,
\item Event $E=4$: detection of an infective by contact-tracing,
\item Event $E=5$: death/emigration of an infective.
\end{itemize}

\noindent Before providing a description of the population process introduced above via a system of SDEs, a few remarks and examples are in order.\\

\noindent {\sc Examples.  (On modelling the contact-tracing feature)} \textit{As previously explained, the removed individuals contribute to contact-tracing in function of the time since their detections and through the weight function $\psi$.\\
1. In the case where the information provided by a detected person enables to examine individuals at a constant rate over a period of time of fixed length $\tau>0$ immediately after its detection (after this time period the information is considered as consumed), the weight function could be chosen as}
\[
\psi(a)=\mathbf{1}_{\{a\in[0,\tau]\}},\;\mbox{for all }a\geq 0.
\]
\textit{Here we have denoted by $\mathbf{1}_{\mathcal{E}}$ the indicator function of $\mathcal{E}$. The second argument of the contact-tracing detection rate, $\langle R_t, \psi\rangle$, is then the number of individuals detected between times $t-\tau$ and $t$}.\\
\textit{2. The following choice:}
\begin{equation}
\psi(a)=e^{-c\cdot a},\;\mbox{for all }a\geq 0,\label{exemplepsi}
\end{equation}
\textit{with $c>0$, is of particular interest when assuming that efficiency of the information provided by a detected individual (geometrically) decreases as the time $a$ since its detection increases.}\\
\textit{3. To take into account the difficulties one may encounter at the early stages of the search for contacts, we can consider functions $\psi$ that are increasing from zero before decaying. From this viewpoint, a suitable beta or gamma density function would be possibly a reasonable choice of parametric weight function $\psi$. }
\begin{remark} {\sc (Explicit forms for jump rate functions)} \label{RKform}
Until now, no explicit form for the infection rate and the detection by contact-tracing rate has been specified for generality's sake (it shall be nevertheless assumed that $\lambda_1$ and $\lambda_3$ both fulfill the collection of assumptions \textbf{H1} listed below). In practice typical choices for the infection rate function are $\lambda_1SI$, $\lambda_1SI/(I+S)$  or $\lambda_1I$ with $\lambda_1>0$ (in order to lighten notation, abusively, we shall still denote by $\lambda_1$ and $\lambda_3$ the parameters characterizing the parametric forms of the rate functions $\lambda_1(.,.)$ and $\lambda_3(.,.,.)$). In the first example, a susceptible becomes infected with an individual rate $\lambda_1 I$ proportional to the number of infected individuals. This rate is generally referred to as the \textit{mass action principle} based model. In contradistinction, the two last examples correspond to a situation where the rate at which a given infective makes infectious contact does not increase with the size $S$ of the population of susceptibles. Equivalently, the individual rate at which a susceptible becomes infected, $\lambda_1 I/S$ or $\lambda_1 I/(I+S)$, decreases when $S$ increases. Such rates are usually termed \textit{frequency dependent}. In a similar fashion, the rate of detection by contact-tracing may be chosen as $\lambda_3 I \langle R, \psi \rangle$, $\lambda_3 I \langle R, \psi \rangle /(I+\langle R, \psi \rangle)$ or $\lambda_3 \langle R, \psi \rangle$, with $\lambda_3>0$.
\end{remark}

\begin{remark} {\sc (A more general framework)} One may consider generalizations of the setup described above, stipulating for instance that the 'S' and 'I' classes are stratified according to socio-demographic features (or sexual behavior characteristics in the context of STD's) in order to account for heterogeneities caused by the social structure of the population, even if it entails introducing more duration variables in the model. One may also introduce the 'age of infection'. This would enable us to model directly time-varying infectivity, offering this way an alternative to so-called 'stage modelling' approaches (see \cite{Isham} for instance). The theoretical results of this paper may be extended in a straightforward manner to such more general frameworks. In order to lighten the notation and make proofs simpler, we restrict the study to the model described above.
\end{remark}

\noindent \textbf{Assumptions H1:} In the remainder of the paper, the rate functions $\lambda_1$ and $\lambda_3$ are assumed to belong to $\Co^1(\R_+^2)$, the set of real functions of class $\Co^1$ on $\R_+^2$. We denote by by $\partial_S \lambda_1$, $\partial_I \lambda_1$, $\partial_I \lambda_3$ and $\partial_R \lambda_3$ their partial derivatives. We also suppose that all these functions are locally Lipschitz continuous and dominated by the mapping $(x,x')\in \mathbb{R}_+^2\mapsto xx'$: for $k\in\{1,3\}$, we assume that $\forall N>0,\, \exists L_k(N)>0$ such that
\begin{equation*}
\forall (x,x'),(y,y')\in [0,N]^2,\;|\lambda_k(x,x')-\lambda_k(y,y')|
\leq L_k(N)(|x-y|+|x'-y'|), \end{equation*}
 and that $\exists \bar{\lambda}_k>0,\,\forall (x,x')\in \R_+^2,\, \lambda_k(x,x')\leq \bar{\lambda}_k  x x'$ (and similarly for the partial derivatives). Finally, the weight function $\psi$ is assumed measurable and bounded. Furthermore, for the central limit theorem, we will suppose it is of class $\Co^2$.

\subsection{On describing the epidemic by stochastic differential equations}

Treading in the steps of \cite{fourniermeleard} who fully developed a microscopic approach for ecological systems (see also \cite{trangdesdev, chithese} where age-structure is taken into account), we now describe the temporal evolution of the epidemic by a measure-valued SDE system driven by Poisson point measures.

The process $\{(S_t,I_t,R_t(da))\}_{t\geq 0}$ defined through the SDE system below takes its values in $\N\times \N \times \mathcal{M}_P(\R_+)$ and may be seen as a generalization of the classical vector-valued Markov processes arising in stochastic SIR models.

\begin{definition}\label{definitionmicroinit}
Consider a probability space $(\Omega,\mathcal{F},\mathbb{P})$, on which are defined:
\begin{enumerate}
\item a random vector $(S_0, I_0)$ with values in $(\N^*)^2$ such that $\mathbb{E}[S_0+I_0]<+\infty$ (at $t=0$, we assume that no one has been detected yet),
\item two independent Poisson point measures on $\R^2_+$, $Q^S(dv,du)$ and $Q^I(dv,du)$, with intensity $dv\otimes  du$, the Lebesgue measure on $\R_+^2$, and independent from the initial conditions $(S_0,I_0)$.
\end{enumerate}
\par Define $\{(S_t,I_t,R_t(da))\}_{t\geq 0}$ as the Markov process solution of the following system of SDEs:
\begin{equation}
\left \{
\begin{array}{lll}
S_t &= & S_0+\int_{v=0}^t \int_{u=0}^{\infty}\left(\mathbf{1}_{0\leq u\leq \lambda_0}-\mathbf{1}_{\lambda_0<u\leq \lambda_0+\mu_0 S_{v-}+\lambda_1(S_{v-},I_{v-})}\right) Q^S(dv,du)\\
I_t &= & I_0+ \int_{v=0}^t \int_{u=0}^{\infty} \mathbf{1}_{\lambda_0 <u\leq \lambda_0 +\lambda_1(I_{v-},S_{v-})}Q^S(dv,du)\\
&-& \int_{v=0}^t \int_{u=0}^{\infty} \mathbf{1}_{0\leq u\leq ( \mu_1+\lambda_2)I_{v-}+\lambda_3(I_{v-},\langle R_{v-},\psi\rangle)}Q^I(dv,du) \\
\langle R_t, f\rangle  & =   &\int_{v=0}^t \int_{u=0}^{\infty}f(0)\mathbf{1}_{0\leq u\leq \lambda_2I_{v-}+\lambda_3(I_{v-},\langle R_{v-},\psi\rangle)} Q^I(dv,du)
+ \int_{v=0}^t \int_{a=0}^{\infty} \partial_a f(a) R_v(da)dv,
\label{descriptionalgoinit}
\end{array}
\right.
\end{equation}
\noindent for all $f\in \Co^1_b(\R_+)$ the set of real bounded functions of class $\Co^1$ with bounded derivatives. We have denoted by $\partial_a f$ the gradient of $f$ and by $g(t-)$ the left limit in $t\in \R$ of any c\`adl\`ag function $g:\R\rightarrow\R$.
\end{definition}
Under \textbf{H1} and (i) of Definition \ref{definitionmicroinit}, it may be seen that there exists a unique strong (non explosive) solution to SDE (\ref{descriptionalgoinit}). The proof is a slight modification of the proofs of Section 2.2 in \cite{chithese}.

\begin{figure}[ht]
\begin{center}\hspace{0cm}
\begin{minipage}[b]{.5\textwidth}\centering
\includegraphics[width=0.9\textwidth,height=0.20\textheight,angle=0,trim=0cm 0cm 0cm 0cm]{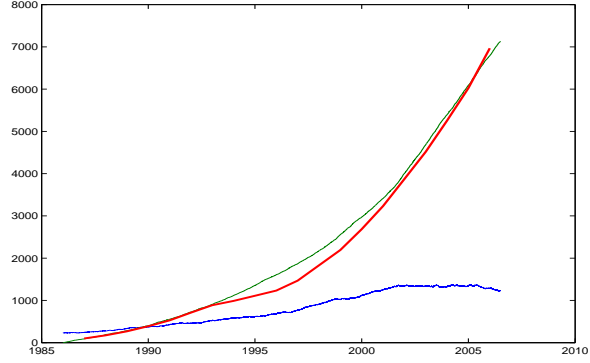}
\end{minipage}
\caption{\textit{Simulations for the Cuban epidemic: simulated evolution of the size of the I class (blue) and of the cumulated size of the R class (green). The bold red line stands for the observed cumulated size of the R class, computed from data related to the Cuban HIV epidemic over the period 1986-2006. We have chosen $\psi=\mathbf{1}_{[0,4]}$. In order to mimic the change in trend that can be observed between 1995 and 2000, two periods have been separately considered. During the first period (\textit{i.e.} the first fifteen years) the parameters of the simulation have been picked as follows: $S_0=5 \,10^6$, $I_0=230$, $\lambda_0=10^{-2}$, $\mu_0=10^{-8}$, $\mu_1=6.6\, 10^{-2}$, $\lambda_1=1.14\, 10^{-7}$, $\lambda_2=3.75 10^{-1}$ and $\lambda_3=6.55 \,10^{-5}$. In the second period, we have used: $\lambda_1=1.16\, 10^{-7}$, $\lambda_2=4.45\,10^{-1}$ and $\lambda_3=2.50\,10^{-4}$.}}\label{figuretrajectoires}
\end{center}
\end{figure}

From a practical perspective, we also emphasize that this approach paves the way for simulating trajectories of the epidemic process (see Fig. \ref{figuretrajectoires}). An attractive advantage of stochastic models in mathematical epidemiology indeed lies in their ability to reproduce certain variability features of the observed data. As an illustration, a simulated trajectory of  $\{(I_t, \langle R_t,\,\mathbf{1}\rangle )\}_{t\geq 0}$ generated from Eq. (\ref{descriptionalgoinit}) with jump rates $\lambda_1SI$ and $\lambda_3 I\langle R,\psi\rangle$ is displayed in Fig. \ref{figuretrajectoires}. For comparison purpose, the observed cumulative number of detected HIV+ individuals in Cuba (1986-2006) has been juxtaposed (see Section \ref{HIV} for further details). It can be seen that over the first 20 years of the epidemic, the simulated and observed curves for the cumulated number of detections are fairly closed to each other, although, between years 7 and 15, the observed curve is slightly below the simulated one (it should be noticed that, during these years, less funds were available for the management of the contact-tracing detection system because of the economic crisis that followed the collapse of the Soviet Union).

\subsection{Limiting behavior in long time asymptotics}
We now state a limit result for the epidemic process introduced above, as time goes to infinity. Refer to \textbf{A1} in the Appendix for a proof based on coupling analysis.

\begin{proposition} \label{ergodicity} Assume that $f(a)\rightarrow 0$ as $a\rightarrow \infty$. Considering the Markov process $\{(S_t,I_t,R_t(da))\}_{t\geq 0}$ introduced in Definition \ref{definitionmicroinit}, we have, whatever the initial conditions $(S_0,I_0)\in (\N^*)^2$, that $
(S_t,I_t,\langle R_t(da),f \rangle)\rightarrow (S_{\infty},0,0)$ in distribution as $t\rightarrow \infty$, denoting by $S_{\infty}$ a Poisson random variable of parameter $\lambda_0/\mu_0$.
\end{proposition}
The law of $S_\infty$ is the stationary distribution of the $\N$-valued immigration and death process which jumps from $k$ to $k+1$ with rate $\lambda_0$ and from $k$ to $k-1$ with rate $\mu_0 k$, and is obtained in Appendix \textbf{A1}. This ergodicity result shows that the time of extinction of the epidemic is almost surely finite, though it may be very long in practice. It is worth mentioning that in the situation of long-lasting epidemics, as in the HIV case, the long term behavior of the epidemic conditioned upon its non extinction may be refined by studying \textit{quasi-stationary measures} (see \cite{vanDoorn} for instance). We leave this question (far from trivial when the state space is not finite) for further research.

\section{Large population limits}
The overall purpose of this section is to provide a thorough analysis of the measure-valued SIR process introduced above from the 'large population approximation' perspective (one may refer to \cite{ethierkurtz} for an account of approximation theorems for Markov processes and to Chapt. 5 in \cite{andersonbritton} for applications of these concepts in approximating vector-valued SIR models), based on the recent techniques developed in \cite{fourniermeleard} and \cite{chithese, trangdesdev} for ecological systems.

\subsection{Renormalization}\label{sectionconstruction}
We consider a sequence, ($\{(S^{(n)}_t,I^{(n)}_t,R^{(n)}_t(da))\}_{t\geq 0}$, $n\in \mathbb{N}^*$), of SIR processes with contact-tracing. For $n\geq 1$, $\{(S^{(n)}_t,I^{(n)}_t,R^{(n)}_t(da))\}_{t\geq 0}$ corresponds to the stochastic process described in Definition \ref{definitionmicroinit}, starting from $(S_0^{(n)},I_0^{(n)})$ of size proportional to $n$ and with the following rate modifications: the immigration rate is $n\lambda_0$, the infection jump rate function is $n\lambda_1(S^{(n)}/n,I^{(n)}/n)$, while  the contact-tracing jump rate function is $n\lambda_3(I^{(n)}/n,\langle R^{(n)}, \psi\rangle/n)$. We denote by $(s^{(n)}_t,i^{(n)}_t,r^{(n)}_t(da))=(S^{(n)}_t/n,I^{(n)}_t/n,R^{(n)}_t(da)/n)$ the renormalized process obtained by re-weighting all individuals of the population by $1/n$. We assume furthermore that $(s^{(n)}_0,i^{(n)}_0)$ converges in probability to a deterministic couple $(s_0,i_0)\in \R_+^{*2}$ as $n\rightarrow \infty$. The moment condition below shall also be required in the sequel. Let $p> 2$.\\

\noindent {\bf Moment assumption $M_p$:  } $\sup_{n\in \N^*}\mathbb{E}[(s^{(n)}_0)^p+(i^{(n)}_0)^p]<+\infty$.\\

\noindent This moment assumption combined with Assumptions \textbf{H1} implies that the moments of order $p$ propagate on compact time intervals $[0,T]$ with $T>0$ (see \cite{fourniermeleard, chithese} Section 3.1.2).
Before writing down the martingale problem associated with $\{(s^{(n)}_t,i^{(n)}_t,r^{(n)}_t(da))\}_{t\geq 0}$ for a given $n\in \N^*$, let us give an insight into the way the renormalizations above may be interpreted in some important examples:
\begin{remark} {\sc (On the meaning of renormalization in basic examples)} In the case of homogeneous rate functions, the eventual impact of the renormalization on the jump rates may be described as follows.
\begin{itemize}
\item With $\lambda_0^{(n)}=n\lambda_0$, the immigration/birth rate is assumed proportional to the initial population size,
\item If the form chosen for $\lambda_1(S,I)$ is either $\lambda_1I$ or $\lambda_1SI/(I+S)$, the renormalized infection rate function, $\lambda_1 I^{(n)}$ or $\lambda_1 S^{(n)}I^{(n)}/(I^{(n)}+S^{(n)})$, is not affected by the scaling, while if one takes $\lambda_1(S,I)=\lambda_1SI$, the renormalized rate function $\lambda_1 S^{(n)}I^{(n)}/n$ decreases proportionately to $1/n$. This reflects the fact that for large scalings (corresponding to large "typical" population sizes) the risk of being contaminated by a given infectious individual is smaller that for small scalings.
\item The same remark holds for the contact-tracing rate function $\lambda_3(I, \langle R, \psi\rangle)$.
\end{itemize}
\end{remark}
\par The next proposition gives a semi-martingale decomposition for $\{(s^{(n)}_t,i^{(n)}_t,r^{(n)}_t(da))\}_{t\geq 0}$, which shall play a crucial role in our analysis.

\begin{proposition}\label{propmartingale}
Let $n\in \N^*$, $t\geq 0$ and $f:(a,u)\mapsto f_u(a)$ a function in $\Co_b^{1}(\R_+^2)$. Under \textbf{H1} and the moment condition \textbf{$M_p$} with $p>2$,

\begin{align}
\left(
\begin{array}{l}
M^{s,(n)}_t\\
M^{i,(n)}_t\\
M^{r,(n)}_t(f)
\end{array}
\right) = & \left(
\begin{array}{l}
s^{(n)}_t-s^{(n)}_0-\lambda_0 t + \int_{u=0}^t \{\mu_0 s^{(n)}_u+\lambda_1(s^{(n)}_u,i^{(n)}_u)\} \,du\\
i^{(n)}_t-i^{(n)}_0-\int_{u=0}^t \{\lambda_1(s^{(n)}_u,i^{(n)}_u) -(\mu_1+\lambda_2)i^{(n)}_u-\lambda_3(i^{(n)}_u,\langle r^{(n)}_u,\psi \rangle)\}du \\
\langle r^{(n)}_t,f_t\rangle - \int_{u=0}^t \{\langle r^{(n)}_u,\partial_a f_u +\partial_u f_u\rangle+f_u(0)(\lambda_2 i^{(n)}_u+\lambda_3(i^{(n)}_u,\langle r^{(n)}_u,\psi\rangle))\}du
\end{array}
\right)\label{martingale}
\end{align}

is a c\`adl\`ag $L^2$-martingale, with predictable quadratic variation given by:

\begin{align}
\left\{
\begin{array}{l}
  \langle M^{s,(n)}\rangle_t =  \frac{1}{n}\int_{u=0}^t \lambda_0+\{\mu_0 s^{(n)}_u+\lambda_1(s^{(n)}_u,i^{(n)}_u)\} \,du\nonumber\\
  \langle M^{i,(n)}\rangle_t =  \frac{1}{n}\int_{u=0}^t \{\lambda_1(s^{(n)}_u,i^{(n)}_u)+ (\mu_1 +\lambda_2) i^{(n)}_u+\lambda_3(i^{(n)}_u,\langle r^{(n)}_u,\psi\rangle)\}du\nonumber\\
  \langle M^{r,(n)}(f)\rangle_t =  \frac{1}{n}\int_{u=0}^t f_u^2(0) \{\lambda_2 i^{(n)}_u+\lambda_3(i^{(n)}_u,\langle r^{(n)}_u,\psi\rangle)\}du\nonumber\\
  \langle M^{s,(n)},M^{i,(n)}\rangle_t =  - \frac{1}{n}\int_{u=0}^t \lambda_1(s^{(n)}_u,I^{(n)}_u) \,du,\quad
\langle M^{s,(n)},M^{r,(n)}(f)\rangle_t = 0\nonumber\\
  \langle M^{i,(n)},M^{r,(n)}(f)\rangle_t =  - \frac{1}{n}\int_{u=0}^t f_u(0) \{\lambda_2 i^{(n)}_u+\lambda_3(i^{(n)}_u,\langle r^{(n)}_u,\psi\rangle)\}du.\label{crochetmn}
 \end{array}\right.
\end{align}
\end{proposition}
This result follows from the representation given in Definition \ref{definitionmicroinit}, in which the Poisson measures $Q^S$ and $Q^I$ are introduced. It may be established by following line by line the proof of Theorem 5.2 of \cite{fourniermeleard} and Theorem 3.1.8 of \cite{chithese}, technical details are thus omitted.

\subsection{Main results for the large population limit}
\subsubsection{Law of large numbers}
Before stating our first limit result for the sequence of renormalized SIR processes introduced above, we make clear the topology we consider. We denote by $\mathcal{M}_F(\R_+)$ the space of finite measures on $\R_+$, endowed with the metrizable weak convergence topology (see \cite{Rachev}). For all $n\in \mathbb{N}^*$, the sample paths $\{(s_t^{(n)},i_t^{(n)},r_t^{(n)})\}_{t\geq 0}$ belong to the Skorohod space $\mathbb{D}(\R_+,\R_+^2\times \mathcal{M}_F(\R_+))$ equipped with the metrizable $J_1$ topology (see $\S$ 2.1 in \cite{joffemetivier} for further details).
\par
Heuristically, since the quadratic variation of the martingale process displayed above is of order $1/n$, one obtains a deterministic limit by letting $n$ tend to infinity. As a matter of fact, consider the system of deterministic evolution equations, obtained by equating to zero the martingale process in Proposition \ref{propmartingale}:
\begin{equation}\label{eqgdepop}
\left \{
\begin{array}{lcl}
s_t&= & s_0+ \int_{u=0}^t \left(\lambda_0  - \mu_0 s_u-\lambda_1(s_u,i_u)\right) du\\
i_t&= & i_0+\int_{u=0}^t \left(\lambda_1(s_u,i_u)-(\mu_1+\lambda_2)i_u-\lambda_3(i_u,\langle r_u,\psi\rangle)\right)du\\
\langle r_t,f_t\rangle &= & \int_{u=0}^t \left\{ \int_{a=0}^{\infty}\left(\partial_u f(a,u)+ \partial_a f(a,u)\right)r_u(da)
+ f(0,u)\left(\lambda_2 i_u+\lambda_3(i_u,\langle r_u,\psi\rangle)\right) \right\}du
\end{array}\right.
\end{equation}
for all $f\in \Co^1_b(\R_+^2)$.
The result below states that there exists a unique (smooth) solution to this deterministic system, to which the sequence $\{(s^{(n)},i^{(n)},r^{(n)}(da))\}_{n\geq 1}$ converges in probability. This may be viewed as an extension to our measure-valued setup of the Law of Large Numbers stated in Theorem 5.2 of \cite{andersonbritton} for vector-valued SIR processes in large population asymptotics (see the references therein). A sketch of proof stands in Appendix \textbf{A2}.
\begin{theorem} {\sc (Law of Large Numbers)}\label{LLN}
Under \textbf{H1} and the moment condition \textbf{$M_p$} with $p>2$, as the size parameter $n$ tends to infinity, the sequence of processes $\{(s^{(n)},i^{(n)},r^{(n)}(da))\}_{n\in \N^*}$ converges in probability in $\mathbb{D}(\R_+,\R_+^2\times \mathcal{M}_F(\R_+))$ to the unique solution $\{(s_t,i_t,r_t(da))\}_{t\geq 0}$ of (\ref{eqgdepop}).
\begin{itemize}
\item[(i)] for all $t>0$, the measure $r_t(da)$ is absolutely continuous with respect to the Lebesgue measure. Denoting by $\rho_t(a)$ its density, the map $(a,t)\mapsto \rho_t(a)$ is differentiable on the set $\{a\leq t\}$ containing its support,
\item[(ii)] the map $t\mapsto(s_t,i_t)$ is of class $\Co^1$.
\end{itemize}
\end{theorem}

By virtue of the regularity properties mentioned above, $(s_t,i_t,\rho_t)_{t\geq 0}$ also solves the following PDE system with initial conditions $(s_0,i_0,\textbf{0})$, $\mathbf{0}$ denoting the constant function equal to zero:

\begin{equation} \label{pdeclassique3}
\left \{
\begin{array}{lcl}
\frac{ds_t}{dt}&= & \lambda_0  - \mu_0 s_t-\lambda_1(s_t,i_t)\\
\frac{di_t}{dt}&= & \lambda_1(s_t,i_t)-(\mu_1+\lambda_2)i_t-\lambda_3\left(i_t,\int_{\R_+}\psi(a)\rho_t(a)da\right)\\
\frac{\partial \rho_t}{\partial t}(a) &=& -  \partial_a \rho_t(a)\\
\rho_t(0)&=&\lambda_2 i_t+\lambda_3\left(i_t,\int_{\R_+}\psi(a)\rho_t(a)da\right).
\end{array}\right.
\end{equation}

This PDE system may be seen as a generalization of deterministic epidemic models introduced in \cite{Hector} (see also the references therein), taking into account the effects of the contact-tracing strategy and defined through a classical differential system. Besides, we point out that the increase of the time since detection ('detection aging') is translated into a transport equation (the third equation in (\ref{pdeclassique3})), with a boundary condition for $a=0$ (fourth equation in (\ref{pdeclassique3})). This is a well-known fact in age-structured population models (see \cite{webb} for instance). It is easy to prove that the solution is of the form $\rho_t(a)=\rho_{t-a}(0)$. One then recovers a \textit{delay-differential equation system} (DDE), similar as those recently considered in epidemic modelling (see \cite{BrauerCastillo, BusenbergCooke, vandenDriessche96, vandenDriessche02, kuang} and the references therein for instance). In this respect, it should be noticed that the DDE system (\ref{pdeclassique3}) may be classically simplified when the weight function $\psi$ is exponential as in Example 2. In this case, it can be replaced by a system of PDEs with finite dimensional variables only.

\subsubsection{Central limit theorem}\label{sectiontcl}

In order to refine the limit result stated in Theorem \ref{LLN}, we establish a central limit theorem (\textit{CLT}), describing how the renormalized epidemic process $(s^{(n)},i^{(n)},r^{(n)}(da))$ fluctuates around the solution of (\ref{pdeclassique3}). This is an adaptation of the results obtained in Chapter 4 of \cite{chithese} for age-structured birth and death processes. Let $T>0$ and consider the sequence of fluctuation processes:
\begin{eqnarray}
\eta_t^{(n)}&= & \left(
\begin{array}{c}
\eta^{s,{(n)}}_t\\
\eta^{i,{(n)}}_t\\
\eta^{r,{(n)}}_t(da)
\end{array}
\right)=\sqrt{n}\left(\begin{array}{c}
s^{(n)}_t-s_t\\
i^{(n)}_t-i_t\\
r^{(n)}_t(da)-r_t(da)
\end{array}\right),\label{deffluctuations}
\end{eqnarray}
$t\in [0,T]$, $n\in \mathbb{N}^*$, with values in $\R\times \R\times \mathcal{M}_S(\R)$, where $\mathcal{M}_S(\R)$ denotes the space of signed measures on $\R$ equipped with its Borel $\sigma$-field.\\

\noindent {\bf Functional preliminaries.} Since $\mathcal{M}_S(\R)$ embedded with the weak convergence topology is not metrizable, we will in fact consider the sequence $(\eta^{r,(n)}(da))_{n\in \N^*}$ as a sequence of processes with values in a well-chosen distribution space. In order to prove its tightness, we link this distribution space to certain Hilbert spaces. We are inspired by the works of M\'etivier \cite{metivierIHP}, M\'el\'eard \cite{meleardfluctuation}, and consider the following spaces:
\begin{definition} \label{deffunctional}
For $\beta \in \N$, $\gamma\in \R_+$, $C^{\beta,\gamma}$ is the space of functions $f$ of class $\mathcal{C}^{\beta}$ such that $\forall k\leq \beta$, $ |f^{(k)}(a)|/(1+|a|^\gamma)$ vanishes as $|a|\rightarrow +\infty$, equipped with the norm:
\begin{equation}
\|f\|_{C^{\beta,\gamma}} :=\sum_{k\leq \beta}\sup_{a\in \R}\frac{|f^{(k)}(a)|}{1+|a|^\gamma}.\label{normec}
\end{equation}The spaces $C^{\beta,\gamma}$ are Banach spaces and we denote by $C^{-\beta,\gamma}$ their dual spaces. \par $W_0^{\beta,\gamma}$ is the closure of the space $\Co^\infty_K(\R)$ of infinitely differentiable functions $f$ with compact support in $\R$ for the norm $\|.\|_{W_0^{\beta,\gamma}}$ defined by:
\begin{equation}
\|f\|_{W_0^{\beta,\gamma}}^2 :=\int_{\R}\sum_{k\leq \beta}\frac{|f^{(k)}(a)|^2}{1+|a|^{2\gamma}} da. \label{normew0}
\end{equation}The spaces $W_0^{\beta, \gamma}$ are Hilbert spaces and we denote by $W_0^{-\beta,\gamma}$ their dual spaces.
\end{definition}
In the following, we will be interested in the following spaces (see \cite{adams, meleardfluctuation} for the continuous injections).
\begin{align}
& C^{4,0}\hookrightarrow W_0^{4,1}\hookrightarrow_{H.S.} W_0^{3,2}\hookrightarrow C^{2,2}\hookrightarrow C^{1,2}\hookrightarrow W_0^{1,3}\hookrightarrow C^{0,4},\nonumber\\
\mbox{and }
 & C^{-0,4}\hookrightarrow W_0^{-1,3}\hookrightarrow C^{-1,2}\hookrightarrow C^{-2,2}\hookrightarrow W_0^{-3,2}\hookrightarrow_{H.S.}W_0^{-4,1}\hookrightarrow C^{-4,0}.\label{plongement2}
\end{align}
The fluctuation processes are now viewed as taking their values in the dual space $C^{-2,2}$ for technical reasons. The space $C^{-2,2}$ is continuously included in $W^{-3,2}_0$ which in turn is included in $W_0^{-4,1}$ by a Hilbert-Schmidt type embedding. It is in this space that the convergence in distribution stated in the next theorem is proved (see Appendix \textbf{A3} for a detailed sketch of the proof).

\begin{theorem}\label{theoremecentrallimiteenonce} {\sc (Central Limit Theorem)} Suppose that \textbf{H1} and the moment assumption \textbf{$M_p$} with $p>2$ are fulfilled and that $\sup_{n\in \N^*}\mathbb{E}\left(|\eta^{s,(n)}_0|^2+|\eta^{i,(n)}_0|^2\right)<+\infty$. Then, ($\{\eta^{(n)}_t\}_{t\in [0,T]})_{n\geq 1}$, as a sequence of random variables with values in $\mathbb{D}([0,T],\R\times \R\times W^{-4,1}_0)$, converges in law to the unique solution of the following equation: $\forall t\in [0,T],$
\begin{equation}
\eta_t=  \eta_0 +W_t+ \int_{u=0}^t \Psi((s_u,i_u,r_u),\;\eta_u)du,\label{evolutionlimitetcl}
\end{equation}
\begin{multline*}
\mbox{where }\quad\Psi((s_u,i_u,r_u),\;\eta_u)\\
=\left(\begin{array}{c}
\mu_0 \eta^s_u+\partial_S\lambda_1(s_u,i_u)\eta^s_u+\partial_I\lambda_1(s_u,i_u)\eta^i_u\\
\partial_S\lambda_1(s_u,i_u)\eta^s_u+[\partial_I\lambda_1(s_u,i_u)+\mu_1+\lambda_2+\partial_I \lambda_3(i_u,\langle r_u,\psi\rangle)] \eta^i_u
+\partial_R \lambda_3(I_u,\langle r_u,\psi\rangle)\langle \eta^r_u,\psi\rangle\\
\delta_0[\lambda_2+\partial_I\lambda_3(i_u,\langle r_u,\psi\rangle)]\eta^i_u+\delta_0
\partial_R\lambda_3(i_u,\langle r_u,\psi\rangle) \langle \eta^r_u,\psi\rangle
+J_u^*\eta^r_u,
\end{array}\right)
\end{multline*}
with $\forall t\in [0,T],\,\forall f\in W_0^{4,1}(\hookrightarrow C^{2,2}),$
$\langle J_t^*\eta^r_t, f\rangle =\int_{\R_+}\partial_a f(a)\eta^r_t(da),
$ and where $W=(W^s, W^i, W^r)$ is a continuous, centered, square-integrable Gaussian process of $\Co([0,T],\R^2\times W_0^{-4,1})$. For every $t\in [0,T]$ and all $f\in W_0^{4,1}$, the quadratic variation of $\left(W^s_t,W^i_t, \langle W^r,f\rangle_t\right)_{t\in [0,T]}$ is given by:
\begin{eqnarray}
\left\{
\begin{array}{ll}
  \langle W^s\rangle_t &=\int_0^t \left(\lambda_0+\mu_0 s_u+\lambda_1(s_u,i_u)\right) du,\\
  \langle W^i\rangle_t &= \int_0^t \left(\lambda_1(s_u,i_u)+(\mu_1+\lambda_2)i_u+\lambda_3(i_u,\langle r_u,\psi\rangle)\right)du\\
  \left\langle   W^{r}(f)\right\rangle_t&=\int_0^t f^2_u(0)\left(\lambda_2 i_u+\lambda_3(i_u,\langle r_u,\psi\rangle)\right)du\\
  \langle W^s,W^i\rangle_t&=-\int_0^t \lambda_1(s_u,i_u)du,\quad \left\langle W^s,\langle W^r,f\rangle\right\rangle_t=0\\
   \left\langle W^i, \langle W^r,f\rangle\right\rangle_t&=-\int_0^t f_u(0)\left(\lambda_2i_u+\lambda_3(i_u,\langle r_u,\psi\rangle)\right)du
\end{array}. \right.\label{crochetlimitetcl}
\end{eqnarray}
\end{theorem}

\section{Statistical inference by maximum likelihood estimation}
We now turn to the problem of estimating the jump rates governing the dynamics of the epidemic, in a parametric setting. Although, generally, not all events related to the epidemic are observable in practice, in this premier work we deal with the ideal case where one dispose of \textit{complete epidemic data} by means of \textit{maximum likelihood estimation} (MLE). Indeed, MLE methods for complete data are of interest from a practical viewpoint, insofar as in certain situations they may be readily used after implementing adequate augmentation data procedures. They constitute, besides, the maximization step of (Monte Carlo-) EM procedures, which are extensively used for analyzing infectious disease data (see \cite{Becker} for instance). However, we stress that the question of validly implementing EM-procedures in a continuous-time process setup is not without pitfalls (see \cite{roberts, beskos} and refer to \cite{ionides} for an account of sequential variants of the MCEM algorithm, tailored for such a framework). Developing inference methods for the present model based on incomplete data (on the incidence process solely, for instance) shall be the scope of further research.\\

\par We start with some definitions. Let us fix the renormalization parameter $n$. We associate to the process $(s^{(n)}_t,i^{(n)}_t,r^{(n)}_t)_{t\geq 0}$ the sequence $\{E_k^{(n)},T_k^{(n)}\}_{k\in \N^*}$ where $\{T^{(n)}_k\}_{k\in \mathbb{N}^*}$ is the sequence of successive jump times of the process, and where $E^{(n)}_k\in\mathcal{E}= \{0,\,\ldots,\,5\}$ is the type of event occurring at time $T^{(n)}_k$, $k\geq 1$ (see Section \ref{dynamics}). By convention, the time origin is $T^{(n)}_0=0$. For notational simplicity only, the rates $\mu_0$, $\mu_1$, $\lambda_0$, $\lambda_1$ are supposed to be known and we focus on the estimation of the detection rates $\lambda_2$ and $\lambda_3$ (extensions to a more general statistical framework are straightforward, in particular when estimating the infection rate is the matter). We suppose that the latter are entirely determined
by a parameter $\theta$, taking values in a set $\Theta \subset \mathbb{R}^d$, $d\geq 1$: $\lambda_2=\lambda_2(\theta)$ and $\lambda_{3}(.,.)=\lambda_{3}(.,.,\theta)$. We set $\{\mathbb{P}_{\theta}\}_{\theta \in \Theta}$
the resulting family of probability measures on the underlying space $(\Omega, \mathcal{A})$. We denote by $\tilde{\mathbb{P}}$ the probability measure
on $(\Omega, \mathcal{A})$ corresponding to the case when the $(E^{(n)}_{k})$'s are i.i.d. and uniformly distributed on $\mathcal{E}$, independent from the durations
$\Delta T^{(n)}_{k}=T^{(n)}_k-T^{(n)}_{k-1}$, $k\in \mathbb{N}$, supposed i.i.d. and exponentially distributed with mean $n/6$.

\subsection{The likelihood function}
Let $T>0$ and $n\in \N^*$. We denote by $K^{(n)}_T=\sum_{k\geq 1}\mathbf{1}_{\{T^{(n)}_k\leq T\}}$ the total number of events occurring before time $T$ and write the likelihood of $\{(E^{(n)}_k,T^{(n)}_k)\}_{1\leq k\leq K^{(n)}_T}$. The complete history of the epidemic until time $T$ is described by the $\sigma$-field $\mathcal{F}^{(n)}_t=\sigma\{s_u^{(n)},i_u^{(n)},r_u^{(n)},\, u\leq t\}$. With the notation above, the statistical model $(\Omega, \mathcal{A}, \{\mathbb{P}_{\theta}\}_{\theta \in \Theta})$ is dominated along the filtration $(\mathcal{F}^{(n)}_t)_{t\geq 0}$ and $\tilde{\mathbb{P}}$ is a dominating probability measure. In particular, for all $\theta\in \Theta$, we have on $\mathcal{F}^{(n)}_T$ that $\mathbb{P}_{\theta}=  \mathcal{L}^{(n)}_T(\theta) \cdot \tilde{\mathbb{P}}$ with the likelihood:
{\small \begin{align}
\mathcal{L}^{(n)}_T(\theta)=& \exp\left(nT - \int_{u=0}^T (n\lambda_0+\mu_0 n s^{(n)}_u+n\lambda_1( s^{(n)}_u, i^{(n)}_u) +(\mu_1 n+\lambda_2(\theta) n) i^{(n)}_u+n\lambda_3 ( i^{(n)}_u,\langle r^{(n)}_u,\psi\rangle,\theta))du\right)\nonumber\\
\times &  \prod_{k=1}^{K_T^{(n)}}L_{\theta}(E_k, (s^{(n)}_{T_k}, i^{(n)}_{T_k},r^{(n)}_{T_k}(da))),
\end{align}}where:
 \begin{equation}L_{\theta}(E,(s, i,r(da))) =  \lambda_0^{\mathbf{1}_{\{E=0\}}} \,
  (\mu_0 s)^{\mathbf{1}_{\{E=1\}}}\,
 \lambda_1(s,i)^{\mathbf{1}_{\{E=2\}}}\,
 (\lambda_2(\theta) i)^{\mathbf{1}_{\{E=3\}}}\,
\lambda_3(i,\langle r, \psi \rangle,\theta)^{\mathbf{1}_{\{E=4\}}}
 (\mu_1 i)^{\mathbf{1}_{\{E=5\}}}.\nonumber
 \end{equation}
If $\theta^*\in \Theta$ denotes the 'true value' of the parameter, by taking the logarithm, keeping the terms depending on $\theta$ solely and using the representation of Definition \ref{definitionmicroinit}, one is lead to maximize the \textit{log-likelihood}:
\begin{align}
l^{(n)}_T(\theta)= &\int_{t=0}^T\int_{u=0}^{\infty}\left[\log(\lambda_2(\theta) i^{(n)}_{t-})\mathbf{1}_{\{0\leq u\leq \lambda_2(\theta^*)ni^{(n)}_{t-}\}}\right.\nonumber\\
+ & \left. \log(\lambda_3 (i^{(n)}_{t-}, \langle r^{(n)}_{t-},\psi\rangle,\theta)\mathbf{1}_{\{\lambda_2(\theta^*) ni^{(n)}_{t-}<u\leq \lambda_2(\theta^*) ni^{(n)}_{t-}+n\lambda_3 (i^{(n)}_{t-},\langle r^{(n)}_{t-},\psi\rangle,\theta^* )\}}\right]Q^I(dt,du)\nonumber\\
-&n\int_{t=0}^T \{\lambda_2(\theta) i^{(n)}_t+\lambda_3 ( i^{(n)}_t,\langle r^{(n)}_t,\psi\rangle,\theta)\}dt. \label{loglik}
\end{align}

\subsection{MLE consistency}Consider the ML estimator for $T>0$ and $n\in \N^*$:
\begin{equation} \label{MLdef}
\hat{\theta}_n=\arg \max_{\theta \in \Theta}l^{(n)}_T(\theta).
\end{equation}
\noindent The following assumptions shall be required:\\

\noindent {\bf Identifiability assumption H2:} The map $\theta \in \Theta \mapsto (\lambda_2(\theta),\;\lambda_3(.,.,\theta))$ is injective.\\

\noindent {\bf Regularity assumption R1:} For all $(x,y)\in \R_+^{*2}$, the maps $\theta\in \Theta \mapsto \lambda_2(\theta)$ and $\theta\in \Theta \mapsto \lambda_3(x,y,\theta)$, are equicontinuous.\\
\par As shown by the result below, under the basic identifiability and regularity conditions stipulated above, ML estimators are consistent.
\begin{theorem} {\sc (Consistency of ML Estimators)} \label{consistency}  Set $\Phi(x)=\log(x)+1/x-1$. Under Assumptions \textbf{H1}, $M_p$ with $p>2$, \textbf{H2} and \textbf{R1}, for all $T>0$ and any $(\theta^*,\theta)\in \Theta^2$, as $n\rightarrow \infty$, we have the following convergence in $\mathbb{P}_{\theta^*}$-probability,
\begin{align}
&  K_n(\theta,\theta^*)=  \frac{1}{n}\{l^{(n)}_T(\theta^*)-l^{(n)}_T(\theta)\}\rightarrow K(\theta,\theta^*),\label{contrast}\\
& \mbox{where: } K(\theta,\theta^*)=\int_{t=0}^T\lambda_2(\theta^*)i^*_t\Phi(\frac{\lambda_2(\theta^*)}{\lambda_2(\theta)})dt
+\int_{t=0}^T\lambda_3(i^*_t,\langle r^*_t,\phi \rangle,\theta^*)\Phi(\frac{\lambda_3(i^*_t,\langle r^*_t,\phi \rangle,\theta^*)}{\lambda_3(i^*_t,\langle r^*_t,\phi \rangle,\theta)})dt,\nonumber
\end{align}
denoting by $(s^*,i^*,r^*(da))$ the solution of the PDE system (\ref{pdeclassique3}) with rate functions associated with $\theta^*$.\\
Under the further assumption that the parameter space $\Theta$ is compact, the ML estimator (\ref{MLdef}) is consistent:
\begin{equation} \label{consist}
\lim_{n\rightarrow \infty}\hat{\theta}_n= \theta^*, \mbox{ in } \mathbb{P}_{\theta^*}- \mbox{probability}.
\end{equation}
\end{theorem}
This result mainly relies on the Law of Large Numbers stated in Theorem \ref{LLN} (see \textbf{A4} in the Appendix for technical details).

\subsection{MLE asymptotic normality}
In order to refine our study of the asymptotic behavior of the ML estimator, we suppose that the stronger regularity assumption below is satisfied.\\

\noindent {\bf Regularity condition R2:} For all $(x,y)\in \R_+^{*2}$, the maps $\theta \in \Theta\mapsto \lambda_2(\theta)$ and $\theta \in \Theta\mapsto \lambda_3(x,y,\theta)$ are twice continuously differentiable.\\

Let $\mathcal{H}_{\theta}g$ denote the hessian matrix of any twice differentiable function $\theta \in \Theta\mapsto g(\theta)$. Observe that the \textit{Fisher information matrix} is given by:
\begin{multline}
\mathcal{I}_{\theta}=-\int_{u=0}^T\left\{\mathcal{H}_{\theta}\lambda_2(\theta) i^*_u\left(\frac{\lambda_2(\theta^*)}{\lambda_2(\theta)}-1\right) - \nabla_\theta \lambda_2(\theta)\cdot ^t\nabla_\theta \lambda_2(\theta)\frac{\lambda_2(\theta^*)i^*_u}{\lambda_2(\theta)^2 } \right.\\
+ \mathcal{H}_{\theta}\lambda_3 ( i^{*}_u,\langle r^{*}_u,\psi\rangle,\theta)\left(\frac{\lambda_3(i^*_u,\langle r^*_u,\psi\rangle,\theta^*)}{\lambda_3(i^*_u,\langle r^*_u,\psi\rangle,\theta)}-1\right) \\
- \left.\nabla_\theta \lambda_3(i^*_u,\langle r^*_u,\psi\rangle,\theta)\cdot ^t\nabla_\theta \lambda_3(i^*_u,\langle r^*_u,\psi\rangle,\theta)\frac{\lambda_3(i^*_u,\langle r^*_u,\psi\rangle , \theta^*)}{\lambda_3(i^*_u,\langle r^*_u,\psi\rangle , \theta)^2} \right\}du.\label{infofisher}
\end{multline}
The next limit result then follows from Theorem \ref{theoremecentrallimiteenonce} (see Appendix \textbf{A5}).
\begin{theorem} {\sc (Asymptotic Normality of ML Estimators)} \label{asymptnorm} Suppose that the assumptions of Theorem \ref{consistency} are fulfilled with the additional condition \textbf{R2}. Then we have the following convergence in distribution under $\mathbb{P}_{\theta^*}$:
\begin{equation}
\sqrt{n}\nabla_{\theta}l_T^{(n)}(\theta^*)\Rightarrow \mathcal{N}(0,\mathcal{I}_{\theta^*}), \mbox{ as } n\rightarrow \infty,
\end{equation}
where $\mathcal{I}_{\theta^*}$ is given by (\ref{infofisher}).
\par Moreover, if $\mathcal{I}_{\theta^*}$ is invertible, then the ML estimator (\ref{MLdef}) is asymptotically normal: under $\mathbb{P}_{\theta^*}$, we have the convergence in distribution
\begin{equation}
\sqrt{n}(\hat{\theta}_n-\theta^*)\Rightarrow \mathcal{N}(0,\mathcal{I}_{\theta^*}^{-1}), \mbox{ as }n\rightarrow \infty.
\end{equation}
\end{theorem}

\subsection{MLE on first simple examples}\label{sectionobscomplete}
Let us consider the important situation where $\lambda_3(i,\langle r,\psi\rangle)$ has a multiplicative form. We will consider the three following models: $\forall i\in \R_+,\,\forall r\in \mathcal{M}_F(\R_+),$
\begin{eqnarray}
 \text{Model (A): }\;\;\lambda_3(i,\langle r,\psi \rangle)&= & \lambda_3 \langle r,\psi\rangle , \label{forme1}\\
  \text{Model (B): }\;\; \lambda_3(i,\langle r,\psi \rangle)&= & \lambda_3 \frac{\langle r,\psi\rangle i}{\langle r,\psi\rangle+i}, \label{forme2}\\
   \text{Model (C): }\;\; \lambda_3(i,\langle r,\psi \rangle)&= & \lambda_3 \langle r,\psi\rangle i. \label{forme3}
\end{eqnarray}
Here $\theta=(\lambda_2,\lambda_3)\in \Theta \subset \R_+^{*2}$, and the true parameter is denoted $\theta^*=(\lambda_2^*,\lambda_3^*)$. In this case, explicit maximum likelihood estimators can be obtained for $\lambda_2$ and $\lambda_3$.
\par We differentiate the log-likelihood with respect to $\lambda_2$ and $\lambda_3$ and obtain \textit{score processes} that we equate to zero. For Model (A) (see Eq. (\ref{forme1})), we have for $n\in \N^*$ and $t\in [0,T]$,
\begin{eqnarray*}
\left \{
\begin{array}{ll}
\frac{\partial l^{(n)}_t}{\partial \lambda_2}(\lambda_2,\lambda_3)= & \int_{v=0}^t\int_{u=0}^{\infty}\frac{1}{\lambda_2}\mathbf{1}_{0\leq u\leq \lambda_2^*n i^{(n)}_{v-}}Q^I(dv,du)
- n \int_{v=0}^t  i^{(n)}_v dv\\
\frac{\partial l^{(n)}_t}{\partial \lambda_3} (\lambda_2,\lambda_3)= & \int_{v=0}^t\int_{u=0}^{\infty}\frac{1}{\lambda_3}\mathbf{1}_{\lambda_2^* n i^{(n)}_{v-}<u\leq \lambda_2^* n i^{(n)}_{v-}+n \lambda_3^* \langle r^{(n)}_{v-},\psi \rangle }Q^I(dv,du)
-n\int_{v=0}^t \langle r^{(n)}_v,\psi\rangle dv,
\end{array}\right.
\end{eqnarray*}
and one gets:
\begin{eqnarray*}
\widehat{\lambda}^{(n)}_2 =  \frac{\int_{v=0}^T\int_{u=0}^{\infty}\mathbf{1}_{0\leq u\leq n\lambda_2^*  i^{(n)}_{v-}}Q^I(dv,du)}{n\int_{v=0}^T  i^{(n)}_v dv},\quad
\widehat{\lambda}_3^{(n,A)} =  \frac{\int_{v=0}^T\int_{u=0}^{\infty}\mathbf{1}_{\lambda_2^* n i^{(n)}_{v-}<u\leq \lambda_2^* n i^{(n)}_{v-}+n \lambda_3^* \langle r^{(n)}_{v-},\psi\rangle }Q^I(dv,du)}{n\int_0^T \langle r^{(n)}_v,\psi\rangle dv}.
\end{eqnarray*}
The Fisher information matrix may be easily explicited in this case:
\begin{eqnarray*}
\mathcal{I}_{\theta^*}&= &  \left(
\begin{array}{cc}
\int_0^T\frac{i^*_t}{\lambda_2^*}dt & 0\\
0 & \sigma^2_3
\end{array}
\right)\quad \mbox{ with }\quad \sigma^2_3=\frac{1}{\lambda_3^*}\int_0^T\langle r^*_t,\psi\rangle dt.
\end{eqnarray*}
\par Now in the case where the contact-tracing detection rate is of the form (\ref{forme2}) or (\ref{forme3}),  the MLE estimate $\widehat{\lambda}_2^{(n)}$ and its asymptotic variance remain both unchanged, while we get for Model (B):
\begin{align*}
\widehat{\lambda}_3^{(n,B)}= & \frac{\int_{v=0}^T\int_{u=0}^{\infty}\mathbf{1}_{\lambda_2^* n i^{(n)}_{v-}<u\leq \lambda_2^* n i^{(n)}_{v-}+n \lambda_3^* i^{(n)}_{v-}\langle r^{(n)}_{v-},\psi\rangle/(i^{(n)}_{v-}+\langle r^{(n)}_{v-},\psi\rangle) }Q^I(dv,du)}{n\int_0^T \frac{i^{(n)}_{v}\langle r^{(n)}_{v},\psi\rangle}{(i^{(n)}_{v}+\langle r^{(n)}_{v},\psi\rangle)} dv}\\
\sigma^2_3= & \frac{1}{\lambda_3^*}\int_0^T \frac{i^*_t \langle r^*_t,\psi\rangle}{(i_t^*+\langle r^*_t,\psi\rangle)}dt,
\end{align*}and for Model (C):
\begin{align*}
\widehat{\lambda}_3^{(n,C)}= & \frac{\int_{v=0}^T\int_{u=0}^{\infty}\mathbf{1}_{\lambda_2^* n i^{(n)}_{v-}<u\leq \lambda_2^* n i^{(n)}_{v-}+n \lambda_3^* i^{(n)}_{v-}\langle r^{(n)}_{v-},\psi\rangle}Q^I(dv,du)}{n\int_0^T i^{(n)}_{v}\langle r^{(n)}_{v},\psi\rangle dv},\quad
\sigma^2_3= \frac{1}{\lambda_3^*}\int_0^T i^*_t\langle r^*_t,\psi\rangle dt.
\end{align*}

\section{Application to HIV data related to the Cuban epidemic (1986-96)} \label{HIV}

This section is devoted to briefly present and discuss preliminary numerical results derived from the application of the statistical modelling previously described to real data. These data are related to the HIV epidemic in Cuba over the period 1986-1996. Our aim is to illustrate the practical interest of the theoretical notions considered in this paper. Owing to space limitations, the statistical issues of validating the parametric model stipulated below for such data shall be thoroughly investigated in a forthcoming paper, entirely dedicated to model checking, estimation and testing. \\

\noindent {\bf The Cuban HIV epidemic.} In prospect of management and analysis of the epidemic, information related to the spread of HIV/AIDS in Cuba has started to be collected and gathered (in a now massive and well-documented data repository) since 1986, after the first HIV cases were detected (see the site of the World Health Organization \cite{worldhealthorg} and refer to \cite{Auvert} for a detailed description of the HIV/AIDS epidemic in Cuba). Each time a person is detected as seropositive (HIV+), the following information is reported: date and way of detection, age, gender, area of residence, gender of sexual partners in the last two years. Furthermore, from the beginning of 1986, all detected persons are invited to give names and contact details of their recent sexual partners. Theses partners are then traced and a recommendation for HIV testing is made over a period of one year after the last sexual contact with the HIV-infected person (see \cite{Hsieh} for further
  details on the Cuban management system of the AIDS epidemic). Based on the information provided, through the interview following HIV detection in the one hand and through the contact-tracing system in the other hand, a plausible date of infection for each individual is reported in the database by the Health authorities (here, as a first go, we shall use these approximate dates as if they were all exact). After being detected, the person receives, either in the sanatoria or in the ambulatory system, regular counselling on living with HIV in order to prevent the risk of transmitting the retrovirus (defending in this respect the assumption that, once detected, a HIV carrier does not belong to the $I$ population any more). Besides, it is essential to notice that one may assert that HIV spreads in Cuba by means of sexual transmission quasi-solely, due to certain distinctive sociological features (see \cite{Auvert}): indeed, since the epidemic began, injection drug use and blood
 transfusion accounted for a negligible number of infections, corresponding to very isolated cases.
\\

\noindent {\bf ML estimation.}
A first attempt is now made to fit the three simple parametric SIR models with contact-tracing described in Section \ref{sectionobscomplete} with the weight function $\psi(a)=\exp(-ca)$ (the age $a$ being given in days and the constant $c$ being a parameter that we will choose (see (\ref{exemplepsi}))).
From the available data, the trajectory $\{(s^{(n)}_t,i^{(n)}_t,r^{(n)}_t(da))\}_{t\in [0,T]}$ may be reconstructed over the 2400 first days after the first detection (more than 6.5 years). Using the results of Section 4.4, we obtain the numerical results of Table \ref{tableresultats}. We have chosen $c=10^{-3}$, $c=10^{-2}$ and $c=3\,10^{-4}$, which correspond to various lifelengths of the information given by a detected patient to trace her/his infectious non detected partners.

\begin{table}[ht]
\centering
\begin{tabular}{|ccccc|}
\hline
model & parameter & estimated value &  asymptotic std & log-likelihood\\
\hline
  & $\lambda_2$ & $9.57\,10^{-4}$ & $4.34\,10^{-5}$ & \\
  \hline
\multicolumn{5}{|c|}{ $c=10^{-2}$ }\\
\hline
 (A) & $\lambda_3$ & $3.90\,10^{-3}$ & $ 2.30\,10^{-4}$ & -2115\\
  (B) & $\lambda_3$ & $4.50\,10^{-3}$ & $ 2.67\,10^{-4}$ & -2119\\
   (C) & $\lambda_3$ & $1.85\,10^{-5}$ & $1.09\,10^{-6}$ & -2117 \\
\hline
\multicolumn{5}{|c|}{ $c=10^{-3}$ }\\
\hline
 (A) & $\lambda_3$ & $6.56\,10^{-4}$ & $ 3.87\,10^{-5}$ & -2138\\
  (B) & $\lambda_3$ & $1.30\,10^{-3}$ & $ 7.72\,10^{-5}$ & -2143\\
   (C) & $\lambda_3$ & $3.11\,10^{-6}$ & $1.83\,10^{-7}$ & -2140 \\
\hline
\multicolumn{5}{|c|}{ $c=3\,10^{-4}$ }\\
\hline
 (A) & $\lambda_3$ & $4.37\,10^{-4}$ & $ 2.58\,10^{-5}$ & -2144\\
  (B) & $\lambda_3$ & $1.10\,10^{-3}$ & $ 6.54\,10^{-5}$ & -2146\\
   (C) & $\lambda_3$ & $2.07\,10^{-6}$ & $1.22\,10^{-7}$ & -2147 \\
   \hline
\end{tabular}\vspace{0.5cm}
\caption{\textit{Estimated parameters and asymptotic standard deviations.}}\label{tableresultats}
\end{table}
Figure \ref{figtauxinstantanes} shows the instantaneous detection rates $t\mapsto \widehat{\lambda}_2^{(n)}i^{(n)}_t$ and $t\mapsto \widehat{\lambda}_3^{(n)}(i^{(n)}_t, \langle r^{(n)}_t,\psi\rangle)$ for the three models (A), (B) and (C).

\begin{figure}[ht]
\begin{center}
\begin{tabular}[!ht]{ccc}
(a) & (b) & (c)\\
\hspace{-0.5cm}\begin{minipage}[b]{.33\textwidth}\centering
\includegraphics[width=0.9\textwidth,height=0.20\textheight,angle=0,trim=0cm 0cm 0cm 0cm]{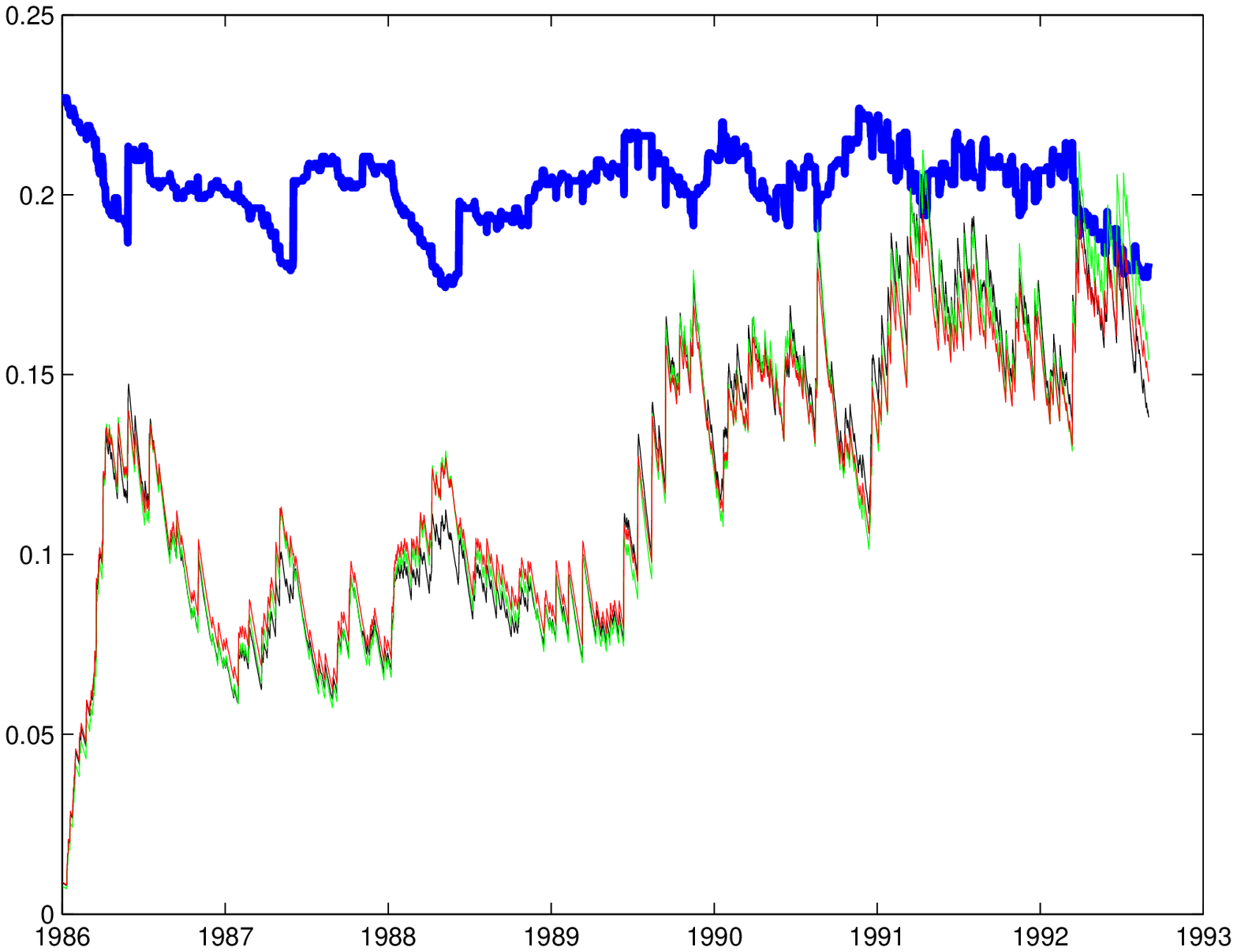}
\end{minipage} &
\begin{minipage}[b]{.33\textwidth}\centering
\includegraphics[width=0.9\textwidth,height=0.20\textheight,angle=0,trim=0cm 0cm 0cm 0cm]{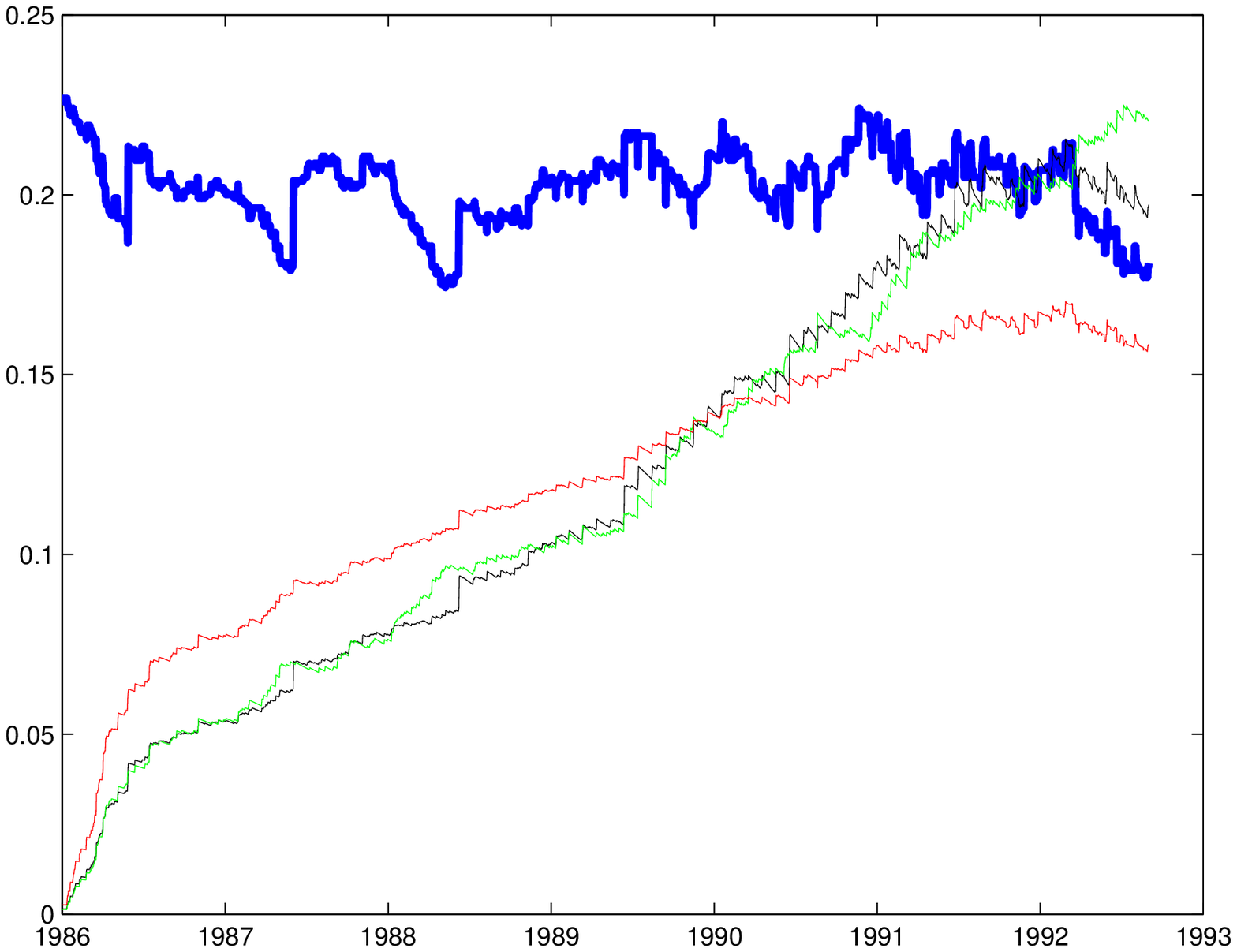}
\end{minipage} &
\begin{minipage}[b]{.33\textwidth}\centering
\includegraphics[width=0.9\textwidth,height=0.20\textheight,angle=0,trim=0cm 0cm 0cm 0cm]{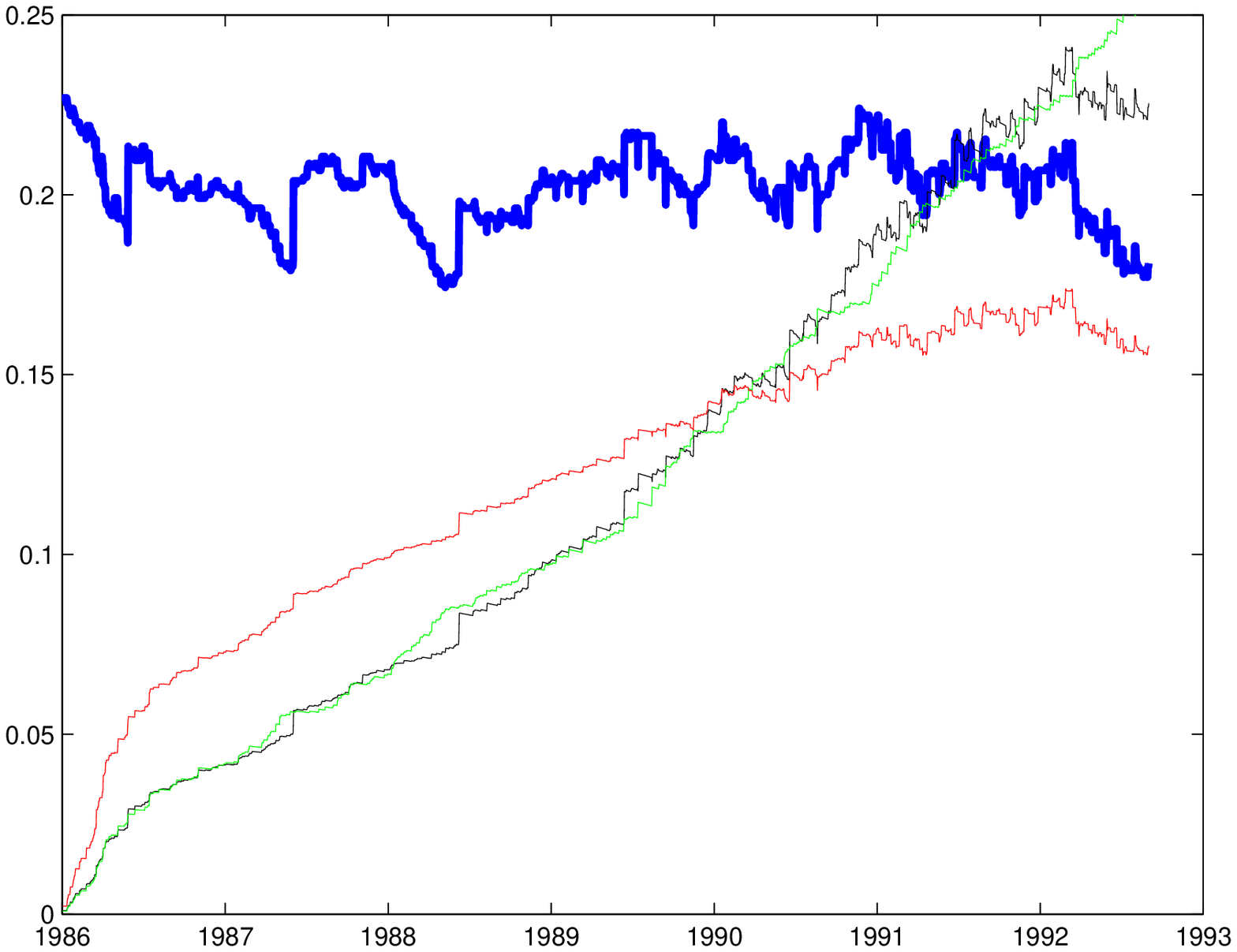}
\end{minipage}
\end{tabular}
\caption{\textit{Evolution of the instantaneous detection rates $t\mapsto \widehat{\lambda}_2^{(n)}i^{(n)}_t$ (in thick blue line) and $t\mapsto \widehat{\lambda}_3^{(n)}(i^{(n)}_t, \langle r^{(n)}_t,\psi\rangle)$ (in green (resp. red and black) line for Model (A) (resp. (B) and (C))). (a) : $c=10^{-2}$, (b) : $c=10^{-3}$, (c) : $c=3\,10^{-4}$}}\label{figtauxinstantanes}
\end{center}
\end{figure}

\par Recall that $\widehat{\lambda}_2$ and its asymptotic variance remain unchanged, whatever the model considered (A), (B) or (C). For the instantaneous rate of contact-tracing detection, we notice that the curves obtained by the models (A), (B) and (C) are very similar. This suggests that the model might be relatively robust to the choice of $\psi$ (provided it has the exponential parametric form stipulated here). The curves for $t\mapsto \widehat{\lambda}_3^{(n,A)}\langle r^{(n)}_t,\psi\rangle$ and $t\mapsto \widehat{\lambda}_3^{(n,C)}\langle r^{(n)}_t,\psi\rangle i^{(n)}_t$ are very close. This is due to the fact that on the considered period, the number of infectious individuals remains stable. Compared with Models (A) and (C), the estimated instantaneous detection rate in Model (B) is more important in the beginning and less important in the end. This is due to the nonlinearity introduced by the denominator in (\ref{forme2}) and to the fact that the number of detected individuals increases with time. The information given by a detected individual to trace new HIV+ ones has a smaller effect when the number of individuals is already high (certain "networks" have already been discovered then). We also underline that the larger $c$ is and the closer the curves for the contact-tracing detection are.
\begin{figure}[ht]
\begin{center}
\begin{tabular}[!ht]{cc}
(a) & (b) \\
\begin{minipage}[b]{.45\textwidth}\centering
\includegraphics[width=0.9\textwidth,height=0.22\textheight,angle=0,trim=0cm 0cm 0cm 0cm]{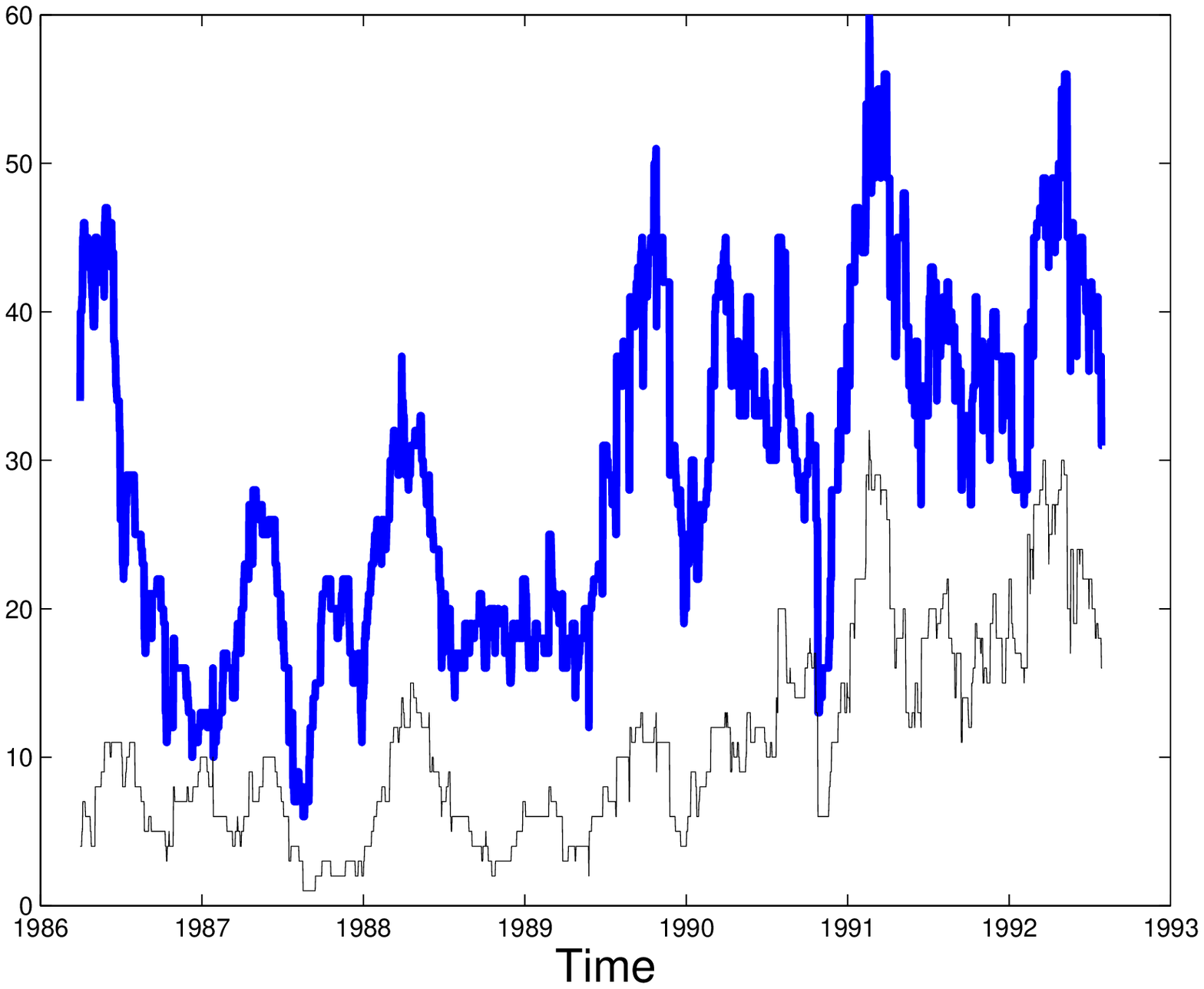}
\end{minipage} &
\begin{minipage}[b]{.45\textwidth}\centering
\includegraphics[width=0.9\textwidth,height=0.225\textheight,angle=0,trim=0cm 0cm 0cm 0cm]{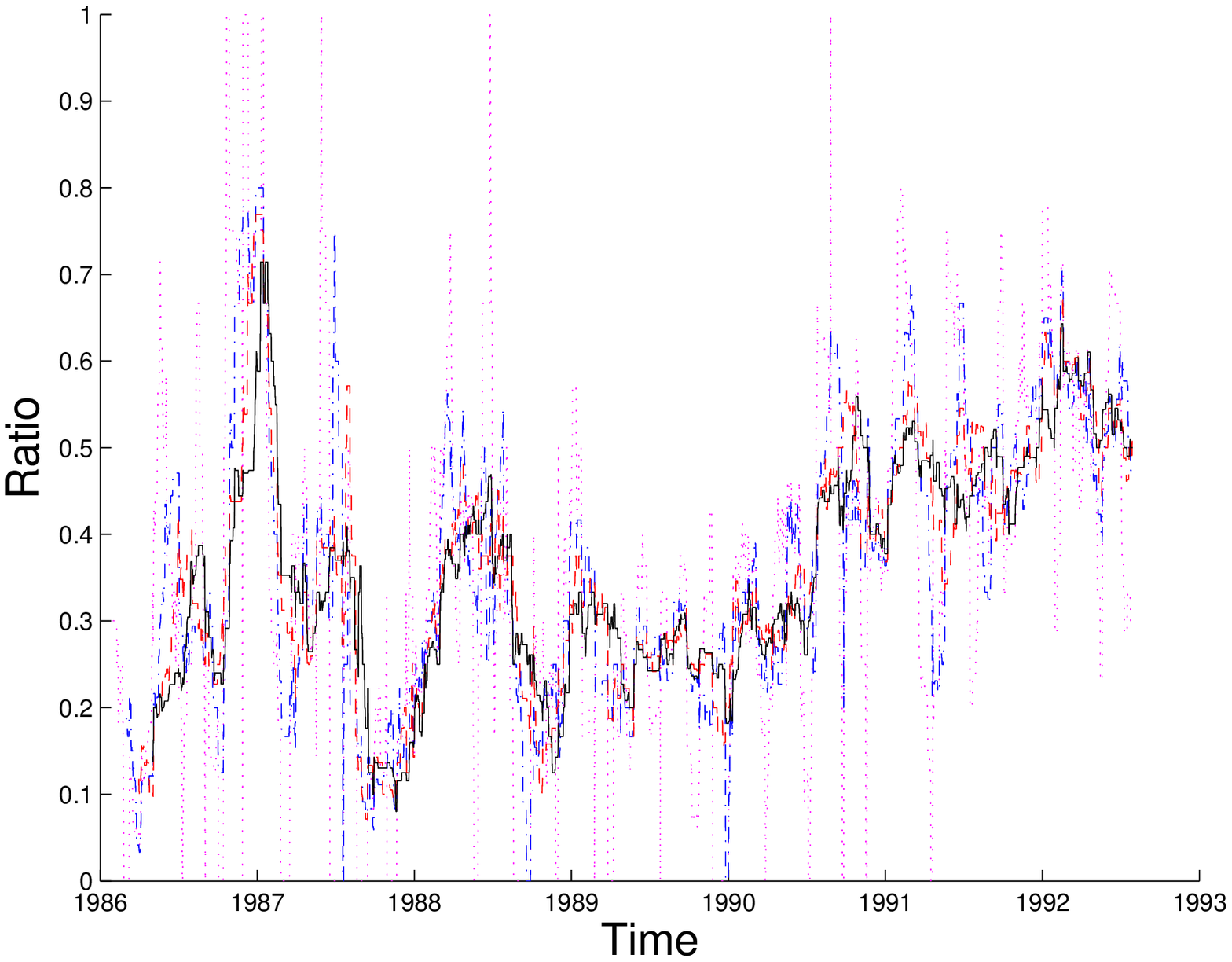}
\end{minipage}\end{tabular}
\caption{\textit{(a): 3-months Moving Average for the number of total detections (thin black line) and for the number of detection by contact-tracing (thick blue line). (b): Percentage of detections by contact-tracing obtained with the moving averages computed over 1 month (dotted magenta), 2 months (dash-dot blue), 3 months (dashed red) and 4 months (solid black).}}\label{figmoyennemob}
\end{center}
\end{figure}

\par On Fig. \ref{figtauxinstantanes}, it can be seen that the rate of detection by contact-tracing increases with time in each of the three models (\ref{forme1})-(\ref{forme3}). It seems that, in all cases, the weight of the contact-tracing detection increases to the point of almost counterbalancing the one of the alternative way of detection. This somehow corroborates the phenomenon underlined by Fig. \ref{figmoyennemob}, in which the proportion of individuals detected by contact-tracing among all detected individuals computed using a moving-window of fixed length is plotted. The graphs displayed in Fig. \ref{figmoyennemob} are model-free and show that the ratio of detections by contact-tracing among all detections stabilizes around 1/2, after having fluctuated during the 6 first years. These years correspond to the "burn-in period" needed for the contact-tracing detection system to start being really efficient.
\par Of course, a very large number of practical questions related to model fitting and interpretation/projection merit further investigation naturally arise, after this premier work. The impact of the choice of the weight function $\psi$ on the obtained results should be carefully investigated for instance. As mentioned above, this shall be the subject of further research, much more oriented towards numerical applications and the practical use of the model.

\section*{Acknowledgements} The authors are grateful to Dr. J. Perez of the National Institute of Tropical Diseases in Cuba for granting them access to the HIV/AIDS database. They acknowledge full support by the French Agency for Research under grant ACI-NIM no. 04-37. They also thank the referees for their judicious comments and helpful remarks, and Daniel Remenik for finding a mistake in a previous version of the paper.

\section*{Appendix - Technical proofs}
\subsection*{A1 Proof of Proposition \ref{ergodicity}} Observe first that, when the Markov epidemic process $(S,I,R(da))$ is restricted to the absorbing set $\mathbb{N}\times \{0\}\times \mathcal{M}_P(\R_+)$, the process $(S_t)_{t\geq 0}$ evolves as an immigration and death process with immigration rate $\lambda_0$ and death rate $\mu_0 S_t$. This process is classically geometrically positive recurrent. Hence, it only remains to show that the set $\mathbb{N} \times \{0\} \times \mathcal{M}_P(\R_+)$ is reached in finite time with probability one, no matter what the initial state. This may be straightforwardly established by coupling analysis: it actually suffices to consider a SIR process $(S',I',R'(da))$ with the same dynamics as $(S,I,R(da))$ except that, when $I=0$ one suppose that new infectives may be recruited from the outside at a strictly positive rate. Then, one may easily see that $(S',I')$ is an irreducible (non explosive) Markov process with state space $\mathbb{N}
 ^2$. Besides, denoting by $\Gamma$ its transition rate matrix, it is straightforward to check that the following Foster-Lyapounov's drift criterion is satisfied, with test function $f(m,l)=m+l$:
\begin{eqnarray*}
\sum_{(m',l')\in \mathbb{N}^2}\Gamma ((m,l),(m',l'))(f(m',l')-f(m,l))&\leq &\lambda_0-\mu_0m-(\mu_1+\lambda_2)l
\leq-cf(m,l)+d,
\end{eqnarray*}
with $c=\min(\mu_0,\mu_1+\lambda_2)$ and $d=\lambda_0$. By virtue of Theorem 7.1 in \cite{meyntweedie}, $(S',I')$ is geometrically recurrent. Thus, when starting from $(S_0,R_0)$, $(S',I')$ and $(S,I)$ reach $\mathbb{N}\times \{0\}$ in a finite time $\tau_{(S_0,R_0)}$ with finite exponential moment. \\

\noindent \textbf{Computation of $S_\infty$'s distribution.} Let us show that $S_\infty$ is a Poisson random variable of parameter $\lambda_0/\mu_0$. Set $p_k:=\mathbb{P}\left(S_\infty=k\right)$ for all $k \in \N$. We have
\begin{align}
\lambda_0 p_0+\mu_0 p_1=0,\quad \mbox{ and }\quad \forall k\geq 1,\, \lambda_0 p_{k-1}-\lambda_0 p_k+\mu_0 (k+1)p_{k+1}-\mu_0 k p_k=0.\label{loisinfty}
\end{align}We obtain from the first equation that $p_1=\lambda_0 p_0/\mu_0$. Assume that we have proved for $k\in \N$ that:
\begin{align}
\forall \ell \leq k,\, p_\ell=\frac{1}{\ell!}\left(\frac{\lambda_0}{\mu_0} \right)^\ell p_0.\label{hyprec}
\end{align}Let us prove that (\ref{hyprec}) holds for $k+1$. From (\ref{loisinfty}) we have:
\begin{align}
p_{k+1}=\frac{p_0}{\mu_0 (k+1)}\left(- \frac{\lambda_0 }{(k-1)!}\left(\frac{\lambda_0}{\mu_0}\right)^{k-1}+\frac{\lambda_0}{k!}\left(\frac{\lambda_0}{\mu_0}\right)^k+\frac{\mu_0 k }{k!}\left(\frac{\lambda_0}{\mu_0}\right)^k\right)=\frac{1}{(k+1)!}\left(\frac{\lambda_0}{\mu_0} \right)^{k+1} p_0.
\end{align}Since the sequence $\{p_k\}$ defines a probability measure, we have $\sum_{k=0}^{+\infty}p_k=1$. This entails that $p_0=e^{-\lambda_0/\mu_0}$ and the desired result then directly follows from the equation above.\\

\subsection*{A2 Proof of Theorem \ref{LLN} (Sketch of)}
This result may be derived from careful examination of Theorems 5.3's proof in \cite{fourniermeleard} or of Theorem 3.2.2's proof in \cite{chithese}. Let $\mathcal{M}_F(\R_+)$ be equipped with the vague convergence topology. Applying Aldous, Rebolledo and Roelly criteria (see \cite{aldous, joffemetivier, roelly}), the sequence $\{(S^{(n)},I^{(n)},R^{(n)})\}_{n\in \N^*}$ is proved tight. By Prohorov's Theorem, there hence exists a subsequence that converges in law to a limiting value $(S,I,R)\in \mathbb{D}(\R_+,\R_+^2\times \mathcal{M}_F(\R_+))$. This subsequence can be chosen such that $\langle R^{(n)},1\rangle$ converges in law to $\langle R,1\rangle$ in $\mathbb{D}(\R_+,\R_+)$. Since the process $\{(S^{(n)}_t,I^{(n)}_t,R^{(n)}_t)\}_{t\geq 0}$ has jumps of amplitude $1/n$, the limiting value is necessarily a continuous process. Using a criterion proposed in \cite{meleardroelly}, one may prove that a subsequence may again be extracted from the previous one, that converges in law to $(S,I,R)$ in $\mathbb{D}(\R_+,\R_+^2 \times \mathcal{M}_F(\R_+))$, where $\mathcal{M}_F(\R_+)$ is endowed with the weak convergence topology this time. The evolution equation is then identified thanks to the martingale representation provided in Proposition \ref{propmartingale}. The quadratic variation vanishes as $n\rightarrow +\infty$ and the moment assumption enables us to take the limit.

\subsection*{A3 Proof of Theorem \ref{theoremecentrallimiteenonce} (Sketch of)}

The proof of Theorem \ref{theoremecentrallimiteenonce} is an adaptation of the argument developed in \cite{metivierIHP}, \cite{meleardfluctuation} and Chapter 4 of \cite{chithese}. We refer the reader to these works for hints to a complete proof and give here the main steps only. Here and throughout, $C(T,N)$ shall denote a constant, depending on $T$ and $N$ only, that will not be necessarily the same at each appearance.\\

\par As previously emphasized, here and throughout the fluctuation process (\ref{deffluctuations}) shall be viewed as a distribution-valued process. We will deal with the spaces introduced in (\ref{plongement2}). The continuous injections which link these spaces are proved in Theorem 5.4 \cite{adams}. We consider (\ref{deffluctuations}) as taking its values in the space $C^{-2,2}$. As the image of $C^{2,2}$ through the differential operator $\partial_a$ is included in $C^{1,2}$, we will be lead to look for estimates in these both spaces. In order to work in a Hilbert setting, we
consider the continuous injections linking this space with $W_0^{-4,1}$ (in which the tightness criterion is proved) and with $W_0^{-1,3}$ (subspace of $C^{-1,2}$ and $C^{-2,2}$ in which the norm of the martingale part of $\eta^n$ is controlled). The continuous embedding $C^{-0,3}\hookrightarrow W_0^{-1,3}$ shall also be required in order to control the norm of certain operators of $W_0^{-1,3}$. The Hilbert-Schmidt embedding $W_0^{-3,2}\hookrightarrow W_0^{-4,1}$ is required by the tightness criteria that we use. Finally, the embedding $W_0^{-4,1}\hookrightarrow C^{-4,0}$ is used to obtain the uniqueness of the limiting value. \\

\noindent\textbf{Fluctuation process.} We have the following decomposition: $\forall f\,:\,(a,t)\mapsto f_t(a)$ in $\Co^{1}(\R_+^2),$

\begin{align}
& \eta^{(n)}_t(f)
 =  \sqrt{n}\left(
\begin{array}{c}
s^{(n)}_0-s_0\\
i^{(n)}_0-i_0\\
0
\end{array}
\right)+\widetilde{M}^{(n)}_t(f)+\widetilde{V}^{(n)}_t(f),\quad \mbox{ where: }\label{tcln}\\
& \eta^{(n)}_t(f)=\left(\begin{array}{c}
\eta^{s,(n)}_t\\
\eta^{i,(n)}_t\\
\langle \eta^{r,(n)}_t,f_t\rangle
\end{array}\right),\,\,\, \widetilde{M}^{(n)}_t(f)=\left(\begin{array}{c}
\widetilde{M}^{s,(n)}_t\\
\widetilde{M}^{i,(n)}_t\\
 \widetilde{M}^{r,(n)}_t(f_t) \end{array}\right)=\sqrt{n}\left(\begin{array}{c}
M^{s,(n)}_t\\
M^{i,(n)}_t\\
 M^{r,(n)}_t(f_t) \end{array}\right),\nonumber\\ 
 & \widetilde{V}^{(n)}_t(f)=
\left(
\begin{array}{c}
\widetilde{V}^{s,(n)}_t\\
\widetilde{V}^{i,(n)}_t\\
 \widetilde{V}^{r,(n)}_t(f_t)
\end{array}
\right)\nonumber
\end{align}
where $ M^{s,(n)},$ $M^{i,(n)}$ and $M^{r,(n)}_t(f_t)$ have been defined in (\ref{martingale}) and where:
\begin{eqnarray*}
\widetilde{V}^{s,(n)}_t&=& \int_{u=0}^t\left\{\mu_0\eta^{s,(n)}_u+\sqrt{n}(\lambda_1(s^{(n)}_u,i^{(n)}_u)-\lambda_1(s_u,i_u))\right\}du\\
\widetilde{V}^{i,(n)}_t&=& \int_{u=0}^t\left\{\sqrt{n}(\lambda_1(s^{(n)}_u,i^{(n)}_u)-\lambda_1(s_u,i_u))+(\mu_1+\lambda_2)\eta^{i,(n)}_u \right. \nonumber\\
&+&\left.\sqrt{n}(\lambda_3(i^{(n)}_u, \langle r^{(n)}_u,\psi\rangle)-\lambda_3(i_u, \langle r_u,\psi\rangle))\right\}du \\
\widetilde{V}^{r,(n)}_t(f_t)&=& \int_{u=0}^t\left\{f_u(0)[\lambda_2 \eta^{i,(n)}_u+\sqrt{n}(\lambda_3(i^{(n)}_u,\langle r^{(n)}_u,\psi\rangle )-\lambda_3(i_u,\langle r_u,\psi\rangle ))]\right. \nonumber\\
&+&\left. \langle \eta^{r,(n)}_u,\,\partial_af_u(a)+\partial_u f_u(a)\rangle \right\}du.
\end{eqnarray*}
{\bf Localization.} A difficulty arises from the fact that the size of the population is not \textit{a priori} bounded. In this respect, for any $N>0$, consider the stopping time:
\begin{equation}
\zeta^{(n)}_N=\inf\left\{t\geq \R_+,\, \max\left(s^{(n)}_t,\, i^{(n)}_t, \, \langle r^{(n)}_t,1\rangle ,\, \langle r^{(n)}_t,|a|\rangle \right)>N\right\}.\label{defzetaNn}
\end{equation}
One may easily see that for all $N>0$ such that
\begin{equation}
N>\max(\sup_{t\in [0,T]} s_t, \sup_{t\in [0,T]}i_t, \sup_{t\in [0,T]}\langle r_t,1\rangle,\sup_{t\in [0,T]}\langle r_t,|a|\rangle).\label{Nachoisir}\end{equation}
we have:
\begin{equation} \lim_{n\rightarrow+\infty}\mathbb{P}\left(\zeta_N^{(n)} \leq  T\right)
=  0\label{localisation}\end{equation}
\noindent\textbf{Moment estimates.} It is straightforward that
$\mathbb{E}[\sup_{t\in [0,T]}\|\eta^{r,(n)}_t\|^2_{W_0^{-1,3}}]\leq C(T, n)$, for all $n\in \N^*$ (see \cite{chithese} Lemma 4.3.1). Since $C(T,n)$ depends on $n$, this estimate is not very interesting. Its importance lies in setting $W_0^{-1,3}$ as a reference space. Using (\ref{plongement2}), we can then also consider it as a process with values in $C^{-1,2}$, $C^{-2,2}$, $W_0^{-3,2}$, $W_0^{-4,1}$ or $C^{-4,0}$. Estimates that do not depend on $n$ can be obtained by following the proofs of Lemmas 4.4.3, 4.4.4 and 4.4.5 in \cite{chithese}:
\begin{lemma}\label{lemmeproptcl}Let $N$ be fixed as in (\ref{Nachoisir}). Then,
\begin{align}
\lefteqn{\sup_{n\in \N^*}\mathbb{E}[\sup_{t\in [0,T\wedge \zeta_N^{(n)}]}\{|\eta^{s,(n)}_t|^2+|\eta^{i,(n)}_t|^2+\|\eta^{r,(n)}_t\|^2_{W_0^{-4,1}}\}]}\nonumber\\
\leq & \sup_{n\in \N^*}\mathbb{E}[\sup_{t\in [0,T\wedge \zeta_N^{(n)}]}\{|\eta^{s,(n)}_t|^2+|\eta^{i,(n)}_t|^2+\|\eta^{r,(n)}_t\|^2_{W_0^{-3,2}}\}]\nonumber\\
\leq & \sup_{n\in \N^*}\mathbb{E}[\sup_{t\in [0,T\wedge \zeta_N^{(n)}]}\{|\eta^{s,(n)}_t|^2+|\eta^{i,(n)}_t|^2+\|\eta^{r,(n)}_t\|^2_{C^{-1,2}}\}]
\leq C(N,T)<+\infty.\label{partie2tclbornebis}
\end{align}
\end{lemma}
\begin{proof}The first two inequalities are consequences of the continuous injection (\ref{plongement2}). Then, (\ref{partie2tclbornebis}) follows from the assumed properties of $\lambda_1$, $\lambda_3$ and $\psi$ combined with the definition of $\|.\|_{C^{-1,2}}$ and the use of Gronwall's Lemma.
\end{proof}

\noindent\textbf{Tightness of the sequence $\{\mathcal{L}(\eta^{(n)})\}_{n\in \N^*}$.} By using a tightness criterion due to M\'etivier (see Section 2.1.5 and Theorem 2.3.2 in \cite{joffemetivier} and Lemma C in \cite{meleardfluctuation}), we will prove thanks to the preceding moment estimates that:
\begin{lemma}\label{proptensionannexe}
The sequence $\{\mathcal{L}(\eta^{(n)})\}_{n\in \N^*}$ of the laws of the fluctuation processes $\{\eta^{(n)}\}_{n\in \N^*}$, when considered as a sequence of $\mathbb{D}([0,T],\R^2\times W_0^{-4,1})$, is tight.
\end{lemma}
\begin{proof}
At first, we have to show that: $\forall t\in [0,T],$
\begin{equation}
\sup_{n}\mathbb{E}\left(|\eta^{s,(n)}_t|^2+|\eta^{i,(n)}_t|^2+ \|\eta^{r,(n)}_t\|^2_{W_0^{-3,2}}\right)<+\infty,\label{etapepoint1tcl}
\end{equation}
where $W_0^{-3,2}$ is a Hilbert space that is embedded in $W_0^{-4,1}$ by a Hilbert-Schmidt embedding. (\ref{etapepoint1tcl}) is a consequence of Lemma \ref{lemmeproptcl}.
\par Then, we establish an Aldous type condition for the finite-variation processes $\{\widetilde{V}^{(n)}\}_{n\in \N^*}$, the quadratic variation processes $\{\langle \widetilde{M}^{s,(n)}\rangle\}_{n\in \N^*}$ and $\{\langle \widetilde{M}^{i,(n)}\rangle\}_{n\in \N^*}$ and the \textit{trace} processes $\{\tlangle \widetilde{M}^{r,(n)}\trangle\}_{n\in \N^*}$ of $\{\widetilde{M}^{r,(n)}\}_{n\in \N^*}$, defined for a Hilbert basis $(\varphi_k)_{k\in \N^*}$ of $W_0^{4,1}$ by: $\forall n\in \N^*,\, \forall t\in [0,T]$,
\begin{eqnarray}
\tlangle \widetilde{M}^{r,(n)} \trangle_t  &= & \int_0^t (\sum_{k\geq 1}\varphi^2_k(0)) \left( \lambda_2 i^{(n)}_u+\lambda_3(i^{(n)}_u,\langle r^{(n)}_u,\psi\rangle)\right)du.\label{tracewidetildem}
\end{eqnarray}
Since $\sum_{k\geq 1}\varphi^2_k(0)\leq C$ (see \cite{meleardfluctuation}), $\tlangle \widetilde{M}^{r,(n)} \trangle_t$ is $\mathbb{P}$-almost surely defined.
\par Let $\delta>0$ and $(S_n,T_n)_{n\in \N^*}$ be a family of stopping times such that $S_n\leq T_n\leq S_n+\delta$. We have: $\forall n\in \N^*,\,\forall \varsigma>0,$
\begin{eqnarray}
  \mathbb{P}\left(\left\|\widetilde{V}^{r,(n)}_{T_n}-\widetilde{V}^{r,(n)}_{S_n}\right\|_{W_0^{-4,1}}\geq \varsigma\right)\leq &  \frac{1}{\varsigma^2}\mathbb{E}\left(\left\|\widetilde{V}^{r,(n)}_{T_n\wedge \zeta^n_N}-\widetilde{V}^{r,(n)}_{S_n\wedge \zeta^n_N}\right\|^2_{W_0^{-4,1}}\right)+\mathbb{P}\left(\zeta^{(n)}_N\leq T\right).\label{etapepoint3tcl}
\end{eqnarray}
To bound the first term, we bound $$
\mathbb{E}\left[\left\|f\mapsto \int_{s}^{t}\int_{\R_+}\partial_a f(a)\eta^{r,(n)}_u(da)du\right\|^2_{C^{-2,2}}\right]$$ 
by $$\mathbb{E}\left[(C \int_s^t \|\eta^{r,(n)}_u\|_{C^{-1,2}} du )^2\right]$$ by noting that $\forall f\in C^{2,2}$, $\partial_a f\in C^{1,2}$ with $\|\partial_a f\|_{C^{1,2}}\leq \|f\|_{C^{2,2}}$, and that $\forall s,t\in [0,T],\,\forall n\in \N^*,$
\begin{eqnarray*}
\left|\int_s^t\int_{\R_+}\partial_a f(a)\eta^{r,(n)}_u(da)du\right| &\leq &  \int_s^t \|\eta^{r,(n)}_u\|_{C^{-1,2}}\|\partial_a f\|_{C^{1,2}} du
\leq     C \int_s^t \|\eta^{r,(n)}_u\|_{C^{-1,2}}\|f\|_{C^{2,2}} du.
\end{eqnarray*}
Using Lemma \ref{lemmeproptcl}, we thus obtain:
\begin{equation}
\mathbb{P} \left(\left\|\widetilde{V}^{r,(n)}_{T_n\wedge \zeta^{(n)}_N}-\widetilde{V}^{r,(n)}_{S_n\wedge \zeta^n_N}\right\|^2_{W_0^{-4,1}}>\varsigma\right) <\frac{C( N)\delta^2}{\varsigma^2}+\mathbb{P}\left(\zeta^{(n)}_N\leq T\right).\label{etapepoint2tcl}
\end{equation}
Similar computations can be carried out for $\widetilde{V}^{s,(n)}$, $\widetilde{V}^{i,(n)}$, $\langle \widetilde{M}^{s,(n)}\rangle$, $\langle \widetilde{M}^{i,(n)}\rangle$ and $\tlangle \widetilde{M}^{r,(n)}\trangle$.  \\

\par With (\ref{etapepoint1tcl}) and (\ref{etapepoint2tcl}), the criterion in Lemma C of \cite{meleardfluctuation} is satisfied and Lemma \ref{proptensionannexe} is proved.
\end{proof}
By virtue of Prohorov's theorem, the sequence $\{\mathcal{L}(\eta^{(n)})\}_{n\in \N^*}$ is relatively compact in $\mathcal{P}(\mathbb{D}([0,T],\R^2\times W_0^{-4,1}))$ embedded with the weak convergence topology. The proof of Theorem \ref{theoremecentrallimiteenonce} is finished by showing that there is a unique adherence value.\\

\noindent \textbf{Identification of the adherence values.} Let $\eta\in \mathbb{D}([0,T],\R^2\times W_0^{-4,1})$ such that $\mathcal{L}(\eta)$ is an adherence value of this sequence. In order to simplify notation, denote again by $(\eta^{(n)})_{n\in \N^*}$ a subsequence that converges in law to $\eta$. Since the magnitude of the jumps of $\eta^{(n)}$ is of order $1/n$, the limiting process $\eta$ is continuous.
\par A first difficulty arises from the fact that Lemma \ref{lemmeproptcl} only deals with the fluctuations localized by the stopping times $\zeta_N^{(n)}$ which depends on $N$.
\par We start off with studying the tightness and the convergence in law of the martingales $(\widetilde{M}^{(n)})_{n\in \N^*}$ for which estimates that do not depend on $n$ nor on $N$ can be established (see \cite{chithese}, Lemma 4.4.5). We can prove, using the same tightness criterion as above, that $(\widetilde{M}^{(n)})_{n\in \N^*}$ is tight in $\mathbb{D}([0,T],\R^2\times W_0^{-4,1})$. Let $W=(W^s,W^i, W^r)$ be continuous martingales as in Theorem \ref{theoremecentrallimiteenonce}: $\forall \varepsilon>0,\,\forall \phi\in W_0^{4,1}$,
\begin{align*}
\lefteqn{  \mathbb{P}(\sup_{t\in [0,T]}\left|\langle \widetilde{M}^{r,(n)}(\phi)\rangle_t - \langle W^r(\phi)\rangle_t\right|>\varepsilon)}\\
\leq&
\frac{1}{\varepsilon}\mathbb{E}[\sup_{t\in [0,T\wedge \zeta^{(n)}_N]}\int_0^t\{\phi^2(0)\lambda_2  \frac{|\eta^{i,(n)}_s|}{\sqrt{n}}+  \phi^2(0)\bar{\lambda}_3 N^2\|\psi\|_{\infty}\frac{|\eta^{i,(n)}_s|}{\sqrt{n}}\\
& + L_3(N)N\frac{|\eta^{i,(n)}_s|+\|\eta^{r,(n)}_s\|_{C^{-1,2}}\|\psi\|_{C^{1,2}}}{\sqrt{n}}
\}ds]\\
& + \mathbb{P}\left(\zeta^{(n)}_N\leq T\right)\\
\leq&  \frac{C(N,T,\|\psi\|_{C^{1,2}})\|\phi\|_\infty}{\varepsilon \sqrt{n}}\mathbb{E}[\sup_{t\in [0,T\wedge \zeta^{(n)}_N]}|\eta^{i,(n)}_s|+\|\eta^{r,(n)}_s\|_{C^{-1,2}}]+\mathbb{P}\left(\zeta^{(n)}_N\leq T\right).
\end{align*}
By (\ref{localisation}) and by Lemma \ref{lemmeproptcl}, this gives that $(\langle\widetilde{M}^{r,(n)}(\phi)\rangle)_{n\in \N^*}$ converges in probability and uniformly in $t\in [0,T]$ to $\langle W^r(\phi)\rangle$, defined in Theorem \ref{theoremecentrallimiteenonce}. Since $\sup_{t\in [0,T]}|\Delta \widetilde{M}^{r,(n)}_t(\phi) |$ is bounded by $C/\sqrt{n}$ and hence uniformly integrable, we obtain by applying Theorem 3.12 page 432 of Jacod and Shiryaev \cite{jacod} that
$(\widetilde{M}^{r,(n)}(\phi))_{n\in \N^*}$ converges in law to the continuous square-integrable gaussian martingale $W^r(\phi)$ starting from 0 and with quadratic variation given by (\ref{crochetlimitetcl}). Similar computations can be done for the other terms of the bracket. Then, $(\widetilde{M}^{n})_{n\in \N^*}$ converges in law in $\mathbb{D}([0,T],\R^2\times W_0^{-4,1})$ to a process $W\in \Co([0,T],\R^2\times W_0^{-4,1})$ such as in Theorem \ref{theoremecentrallimiteenonce}.\\

\par To characterize now the limit value $\eta$, we introduce the following functional, for $\nu=(\nu^s,\nu^i,\nu^r)\in \mathbb{D}([0,T],\R^2\times W_0^{-4,1})$, $\phi\in W_0^{4,1}$ and $t\in [0,T]$:
{\small \begin{align} & \Psi(\nu,\phi,t)=
\left(
\begin{array}{c}
\nu_t^s\\
\nu^i_t\\
\nu^{r}_t(\phi)\end{array}\right)
-\left(
\begin{array}{c}
\eta_0^s\\
\eta^i_0\\
0\end{array}\right)\label{grandpsitcl}\\
 & -    \int_0^t\left(
\begin{array}{c}
-\left(\mu_0 +\partial_S\lambda_1(s_u,i_u)\right)\nu^s_u
-\partial_I\lambda_1(s_u,i_u) \nu^i_u\\
\partial_S\lambda_1(s_u,i_u)\nu^s_u-
\left(\partial_I\lambda_1(s_u,i_u)+\mu_1+\lambda_2
+\partial_I\lambda_3(i_u,\langle r_u,\psi\rangle)\right)\nu^i_u
-\partial_R \lambda_3(i_u,\langle r_u,\psi\rangle) \langle \nu^r_u,\psi\rangle
\\
\phi(t-u)(\lambda_2+\partial_I \lambda_3(i_u,\langle r_u,\psi\rangle))\nu_u^i+\phi(t-u)\partial_R\lambda_3(i_u,\langle r_u,\psi\rangle)\langle \nu_u^r,\psi\rangle\end{array}\right)
du\nonumber
\end{align}}
Notice that, in the definition of $\Psi$, the density dependence has been "frozen".
We can show that the process of $\Co([0,T],\R^2\times W_0^{-4,1})$ defined for every $ \phi\in W_0^{4,1}$ and $t\in [0,T]$ by $\left(\widetilde{M}_t^s,\,
\widetilde{M}^i_t,\,
\widetilde{M}^{r}_t(\phi)\right):=  \Psi(\eta,\phi,t)$ has the same law as the process $W$ defined in Theorem \ref{theoremecentrallimiteenonce}.
\par Indeed, since $(\eta^{(n)})_{n\in \N^*}$ converges in law to the continuous process $\eta$, we have $$\forall \phi\in W_0^{4,1},\,\lim_{n\rightarrow +\infty}\Psi(\eta^{(n)},\phi,.)=\Psi(\eta,\phi,.).$$ We shall now prove that $\Psi(\eta^{(n)},\phi,.)$ has the same limit in law as $\widetilde{M}^{(n)}(\phi)$.

\par From (\ref{deffluctuations}) and (\ref{grandpsitcl}): $\forall n\in \N^*,\,\forall \phi\in W_0^{4,1},\,\forall t\in [0,T],$
\begin{eqnarray}
\lefteqn{|\Psi(\eta^{(n)},\phi,t)-  (\widetilde{M}^{s,(n)}_t,\widetilde{M}^{i,(n)}_s,\widetilde{M}^{r,(n)}_t(\phi))|^2}\nonumber\\
&=&  \left(\int_0^t A(n,s)ds\right)^2+\left(\int_0^t \left[A(n,s) + B(n,\phi,s)\right]ds\right)^2\nonumber\\
&+&  \left(\int_0^t \phi(t-s) B(n,\phi,s)ds\right)^2,\label{etapeidentification1}
\end{eqnarray}
where:
\begin{eqnarray*}
 A(n,u)&= & \partial_S\lambda_1(s_u,i_u)\eta^{s,(n)}_u
+\partial_I\lambda_1(s_u,i_u) \eta^{i,(n)}_u-\sqrt{n}\{\lambda_1(s^{(n)}_u,i^{(n)}_u)-\lambda_1(s_u,i_u)\},\label{rajouta}\\
B(n,s)&= & \{\partial_I\lambda_3(i_u,\langle r_u,\psi\rangle)\eta^{i,(n)}_u
-\partial_R \lambda_3(i_u,\langle r_u,\psi\rangle) \langle \eta^{r,(n)}_u,\psi\rangle\}\nonumber\\
& -  & \sqrt{n}\{\lambda_3(i^{(n)}_u,\langle r^{(n)}_u,\psi\rangle)-\lambda_3(i_u,\langle r_u,\psi\rangle)\}.\label{rajoub}
\end{eqnarray*}
We obtain from
\begin{multline*}
\lambda_1(s^{(n)}_u,i^{(n)}_u)- \lambda_1(s_u,i_u)= \lambda_1(s_u+\frac{\eta^{s,(n)}_u}{\sqrt{n}},i_u+\frac{\eta^{i,(n)}_u}{\sqrt{n}})-\lambda_1(s_u,i_u)\\
=  \int_0^1 \{\partial_S \lambda_1(s_u+\alpha \frac{\eta^{s,(n)}_u}{\sqrt{n}},i_u+\alpha \frac{\eta^{i,(n)}_u}{\sqrt{n}})\frac{\eta^{s,(n)}_u}{\sqrt{n}}+\partial_I \lambda_1(s_u+\alpha \frac{\eta^{s,(n)}_u}{\sqrt{n}},i_u+\alpha \frac{\eta^{i,(n)}_u}{\sqrt{n}})\frac{\eta^{i,(n)}_u}{\sqrt{n}}\}d\alpha,
\end{multline*}
and from the Lipschitz properties of $\partial_S \lambda_1$ and $\partial_I\lambda_1$ that $| A(n,u)|
\leq
C(s_u,i_u)(|\eta^{s,(n)}_u|^2+|\eta^{i,(n)}_u|^2)/\sqrt{n}$ and:
\begin{eqnarray}
\mathbb{P}\left(\sup_{t\in [0,T]}\left|\int_0^t A(n,u) du\right|>\varepsilon\right)
&\leq & \frac{CN\mathbb{E}\left(\sup_{u\in [0,T\wedge \zeta^n_N]}|\eta^{s,(n)}_u|^2+|\eta^{i,(n)}_u|^2\right)}{\varepsilon\sqrt{n}}+\mathbb{P}\left(\zeta^n_N\leq T\right),\label{majorationans}
\end{eqnarray}
which tends to zero as $n\rightarrow \infty$, by virtue of Lemma \ref{lemmeproptcl} and of (\ref{localisation}). We deal with the term $B(n,s)$ with similar computations and obtain that $\int_0^t B(n,s) ds$ converges in probability to 0 uniformly in $t\in [0,T]$. As a consequence, $\Psi(\eta^{(n)},\phi,.)$ converges uniformly in $t\in [0,T]$ and in probability to the same limit as $(\widetilde{M}^{s,(n)}, \widetilde{M}^{i,(n)}, \widetilde{M}^{r,(n)}(\phi))$, which shows that the limiting values $\eta$ satisfy (\ref{evolutionlimitetcl}).\\

\par In order to complete Theorem \ref{theoremecentrallimiteenonce}'s proof, we establish that the strong uniqueness property holds for (\ref{evolutionlimitetcl}) (for given $W$ and $\eta_0$). For this, we work in the space $\Co([0,T],\R^2\times C^{-4,0})$. These results rely on the use of Gronwall's lemma and on the fact that when the density dependence is 'frozen' ($s^{(n)}$, $i^{(n)}$ or $r^{(n)}(da)$ have been replaced by their deterministic limits) the constants appearing in the estimates do not depend on the localization in $N$ any more.

\par Consequently, the adherence value of $(\mathcal{L}(\eta^{(n)}))_{n\in \N^*}$ is unique and the sequence $\eta^{(n)}$ converges in law in $\mathbb{D}([0,T],\R^2\times W^{-4,1})$ to the solution of SDE (\ref{evolutionlimitetcl}).
\subsection*{A4 Proof of Theorem \ref{consistency}}
Using representation (\ref{loglik}), the convergence (\ref{contrast}) directly results from Theorem \ref{LLN} combined with the assumed smoothness properties of $\lambda_1$, $\lambda_3$. Now, by virtue of the identifiability assumption, the limiting contrast $K(\theta,\theta^*)$ equals to $0$ in the sole case where $\theta=\theta^*$ and (\ref{consist}) then follows from the regularity assumption \textbf{R1} in a standard fashion (see \cite{Ibragimovbook}).

\subsection*{A5 Proof of Theorem \ref{asymptnorm}}

Observe first that the map $\theta\in \Theta \mapsto l_T^{(n)}(\theta)$ is twice differentiable, and denoting by $\theta^*\in \Theta$ the true value of the parameter, for all $\theta \in \Theta$ and $T>0$, the \textit{score} $\nabla_{\theta}l_T^{(n)}(\theta)$ equals to
\begin{eqnarray*}
&&\int_{t=0}^T\int_{u=0}^{\infty}\frac{\nabla_{\theta} \lambda_2(\theta)
}{\lambda_2(\theta)}\mathbf{1}_{\{0\leq u\leq n\lambda_2(\theta^*) i^{(n)}_{t-}\}}Q^I(dt,du)\\
&+&\int_{t=0}^T\int_{u=0}^{\infty}\frac{\nabla_{\theta} \lambda_3(i^{(n)}_{t-}, \langle r^{(n)}_{t-}, \psi \rangle, \theta)
}{\lambda_3(i^{(n)}_{t-}, \langle r^{(n)}_{t-}, \psi \rangle, \theta)}\mathbf{1}_{\{\lambda_2(\theta^*) ni^{(n)}_{t-}<u\leq \lambda_2 ni^{(n)}_{t-}+n\lambda_3 (i^{(n)}_{t-},\langle r^{(n)}_{t-},\psi\rangle,\theta^* )\}}Q^I(dt,du)\\
&-&n\int_{u=0}^T\{\nabla_{\theta}\lambda_2(\theta^*) i^{(n)}_{u}+ \nabla_{\theta}\lambda_3(i^{(n)}_{u}, \langle r^{(n)}_{u}, \psi \rangle, \theta)\}du\end{eqnarray*}
Let us define $I^{(n)}_{\theta^*}= -\mathcal{H}_{\theta}l_T^{(n)}(\theta)$. We have
{\small \begin{multline*}
I^{(n)}_{\theta^*}=-n\int_{u=0}^T\{\mathcal{H}_{\theta}\lambda_2(\theta) i^{(n)}_u\\
+\mathcal{H}_{\theta}\lambda_3 ( i^{(n)}_u,\langle r^{(n)}_u,\psi\rangle,\theta)\}du
+\int_{t=0}^T\int_{u=0}^{\infty}\left[\frac{\nabla_{\theta}\lambda_2 (\theta)\cdot ^t\nabla_\theta \lambda_2(\theta)}{\lambda_2 (\theta)}i^{(n)}_{t-}\mathbf{1}_{\{0\leq u\leq n\lambda_2 i^{(n)}_{t-}\}}\right.\\
+\frac{\nabla_\theta\lambda_3 (i^{(n)}_{t-}, \langle r^{(n)}_{t-},\psi\rangle,\theta)\cdot ^t\nabla_\theta\lambda_3 (i^{(n)}_{t-}, \langle r^{(n)}_{t-},\psi\rangle,\theta)}{\lambda_3 (i^{(n)}_{t-}, \langle r^{(n)}_{t-},\psi\rangle,\theta)}\\
\times \left.\mathbf{1}_{\{\lambda_2(\theta^*) ni^{(n)}_{t-}<u\leq \lambda_2(\theta^*) ni^{(n)}_{t-}+n\lambda_3 (i^{(n)}_{t-},\langle r^{(n)}_{t-},\psi\rangle,\theta^* )\}}\right]Q^I(dt,du).
\end{multline*}
}
We have the following result.
\begin{lemma} Under the assumptions of Theorem (\ref{asymptnorm}), we have:
\begin{itemize}
\item[(i)] $I^{(n)}_{\theta^*}\rightarrow \mathcal{I}_{\theta^*}$ in $\mathbb{P}_{\theta^*}$-probability, as $n\rightarrow \infty$,\\
\item[(ii)] for all $T>0$, the sequence of processes $(\{n^{-1/2}\nabla_{\theta}l_t^{(n)}(\theta)\}_{t\in [0,T]}, \,n\in \N^*)$, converges in law in $\mathbb{D}(\R_+,\R)$ to the continuous gaussian martingale process with 0 as initial value and quadratic variation given by:
\begin{eqnarray}
 \int_0^t\left\{\frac{\nabla_\theta \lambda_2(\theta^*)\cdot ^t\nabla_\theta \lambda_2(\theta^*)}{ \lambda_2(\theta^*)}i^*_u+ \frac{\nabla_\theta \lambda_3(i^*_u,\langle r^*_u,\psi\rangle, \theta^*)\cdot ^t\nabla_\theta \lambda_3(i^*_u, \langle r^*_u,\psi\rangle,\theta^*)}{ \lambda_3(i^*_u, \langle r^*_u,\psi\rangle, \theta^*)}\right\}ds
\end{eqnarray}
\end{itemize}
\end{lemma}

\begin{proof}
The first assertion follows from Theorem \ref{LLN}. The second one is a consequence of Theorem \ref{theoremecentrallimiteenonce}.
\end{proof}

The argument of the asymptotic normality results may be classically derived from the lemma above (see Chapter 4 in \cite{Linkov} for instance). Technical details are omitted.

\vspace{0cm}

{\footnotesize
\bibliographystyle{amsplain}
\bibliography{biblioanr}
}
\end{document}